\documentclass{amsart}

\usepackage{amsthm}
\usepackage{graphicx}
\usepackage{stmaryrd} %has \mapsfrom, \llbracket, \rrbracket
\usepackage{tikz}
\usepackage{tikz-cd}
\usetikzlibrary{arrows}
\usepackage{ bbold }
\usepackage{verbatim}
\usepackage{amssymb,amsfonts}
\usepackage[all,arc]{xy}
\usepackage{enumerate}
\usepackage{mathrsfs, mathabx}
\usepackage{xcolor}[2007/01/21]
\usepackage{hyperref}
\usepackage{soul}
\usepackage{esint} % \varointctrclockwise
\usepackage[all]{xy}
\usepackage[margin=1.0in]{geometry}
\usepackage[
   open,
   openlevel=2,
   atend
 ]{bookmark}[2011/12/02]
\usepackage{graphics}

\bookmarksetup{color=blue}

\theoremstyle{definition} 
\newtheorem{thm}{Theorem}[section]

\newtheorem{prop}[thm]{Proposition}
\newtheorem{lem}[thm]{Lemma}
\newtheorem{conj}[thm]{Conjecture}

\newtheorem{setting}[thm]{Properties}

\theoremstyle{definition}
\newtheorem{defn}[thm]{Definition}

\newtheorem{exmp}[thm]{Example}

\newtheorem{notn}[thm]{Notation}

\newtheorem{rmk}[thm]{Remark}

\theoremstyle{definition}

\theoremstyle{definition}

\theoremstyle{remark}

\newcommand{\im}{\mathrm{im}}

\newcommand{\GL}{\mathrm{GL}}

\newcommand{\Aut}{\mathrm{Aut}}

\newcommand{\bZ}{\mathbb{Z}}

\newcommand{\bR}{\mathbb{R}}

\newcommand{\bF}{\mathbb{F}}
\newcommand{\bE}{\mathbb{E}}

\newcommand{\Zp}{\mathbb{Z}_p}
\newcommand{\bb}{\mathbb}

\newcommand{\mr}{\mathrm}
\newcommand{\mf}{\mathfrak}
\newcommand{\ms}{\mathscr}

\newcommand{\sm}{\smallsetminus}
\newcommand{\sub}{\subset}

\newcommand{\sups}{\supset}

\newcommand{\Mat}{\mathrm{Mat}}
\newcommand{\Prob}{\mathrm{Prob}}

\newcommand{\Sur}{\mathrm{Sur}}

\newcommand{\M}{\mr{M}}

\newcommand{\rr}{\mf{r}}
\newcommand{\m}{\mf{m}}

\newcommand{\llb}{\llbracket}
\newcommand{\rrb}{\rrbracket} 
 
\newcommand{\Mod}{\textbf{Mod}}

\newcommand{\ka}{\kappa}
\newcommand{\eps}{\epsilon}
\newcommand{\ep}{\epsilon}

\newcommand{\sg}{\sigma}
\newcommand{\dt}{\delta}
\newcommand{\Hom}{\mathrm{Hom}}
\newcommand{\Ext}{\mathrm{Ext}}

\newcommand{\ol}{\overline}

\newcommand{\ra}{\rightarrow}

\newcommand{\tra}{\twoheadrightarrow}

\newcommand{\lt}{\left}
\newcommand{\rt}{\right}

\newcommand{\ot}{\otimes}

\newcommand{\cok}{\mathrm{cok}}
\newcommand{\Ha}{\mathrm{Haar}}

\newcommand{\diag}{\mathrm{diag}}

\newcommand{\be}{\begin{enumerate}}
\newcommand{\ee}{\end{enumerate}}

\newcommand{\bi}{\begin{itemize}}
\newcommand{\ei}{\end{itemize}}

\newcommand{\bbm}{\begin{bmatrix}}
\newcommand{\ebm}{\end{bmatrix}}

\numberwithin{equation}{section}

\newcommand{\abs}[1]{\lvert#1\rvert}

\usepackage{mathtools}

\DeclarePairedDelimiter{\set}{\{}{\}}

\DeclarePairedDelimiter{\parens}{\lparen}{\rparen}

\newcommand\subeq{\subseteq}

\newcommand{\ls}{\leqslant}

\newcommand\bbar{\overline} %\bar is too narrow; avoid
\newcommand\hhat{\widehat} %\hat is too narrow; avoid
\newcommand\onto{\twoheadrightarrow}

\newcommand{\map}[1][]{\overset{#1}{\to}}
\newcommand{\Fq}{{\mathbb{F}_q}}
\newcommand{\Fp}{{\mathbb{F}_p}}

\newcommand{\Athat}{\hhat{A[t]}}

\DeclareMathOperator{\Tor}{Tor}

\newcommand{\fa}{{\mf a}}
\newcommand{\fb}{{\mf b}}

\newcommand{\ppmod}{\hspace{-1mm} \pmod}

\begin{document}

\title[The cokernel of a random integral matrix with concentrated residue]{The cokernel of a polynomial push-forward of a random integral matrix with concentrated residue}
\date{\today}
\author{Gilyoung Cheong and Yifeng Huang}
\address{G. Cheong -- Department of Mathematics, University of California--Irvine, 340 Rowland Hall, Irvine, California 92697, the United States of America \newline
Y. Huang -- Department of Mathematics, University of British Columbia, 1984 Mathematics Road, Vancouver, BC Canada V6T 1Z2}
\email{gilyounc@uci.edu, huangyf@math.ubc.ca}

\begin{abstract}
We prove new statistical results about the distribution of the cokernel of a random integral matrix with a concentrated residue. Given a prime $p$ and a positive integer $n$, consider a random $n \times n$ matrix $X_n$ over the ring $\bZ_p$ of $p$-adic integers whose entries are independent. Previously, Wood showed that regardless of the distribution of $X_n$, as long as each entry of $X_n$ is not too concentrated on a single residue modulo $p$, the distribution of the cokernel $\cok(X_n)$ of $X_n$, up to isomorphism, weakly converges to the Cohen--Lenstra distribution, as $n \ra \infty$. In this paper, we consider the case when $X_n$ has a concentrated residue $A_n$ so that $X_n = A_n + pB_n$, where $B_n$ is a random $n \times n$ matrix over $\bZ_p$. We show that for every fixed $n$ and a non-constant monic polynomial $P(t) \in \bZ_p[t]$, we can explicitly compute the distribution of $\cok(P(X_n))$ when $B_n$ is a Haar-random matrix. Using this, we also show that for specific choices of $A_n$ a much wider class of random matrices $B_n$ gives the same distribution of $\cok(P(X_n))$. For the Haar-random $B_n$, we deduce our result from an interesting equidistribution result for matrices over $\bZ_p[t]/(P(t))$, which we prove by establishing a version of the Weierstrass preparation theorem for the noncommutative ring $\M_n(\bZ_p)$ of $n \times n$ matrices over $\bZ_p$.
\end{abstract}

\maketitle

\section{Introduction}

\hspace{3mm} Fix a prime $p$ and consider the distribution of the cokernel $\cok(X)$ of a random $n \times n$ matrix $X$ over the ring $\bZ_p$ of $p$-adic integers, where $n \in \bZ_{\geq 1}$. We consider $X$ with $n^2$ independent entries $(X_{ij})_{1 \leq i, j \leq n}$. Writing $\M_n(R)$ to mean the set of $n \times n$ matrices over a ring $R$, we can identify $\M_{n}(\bZ_p) = \bZ_{p}^{n^2}$, and the probability measure on $\M_{n}(\bZ_p)$ is given by the product measure of the probability measures on $n^2$ copies of $\bZ_p$.

\hspace{3mm} Each independent entry $X_{ij}$ of a random matrix $X$ can be written as
\begin{equation}\label{digits}
X_{ij} = X_{i,j,0} + X_{i,j,1} p + X_{i,j,2} p^2 + \cdots
\end{equation}

whose $p$-adic digits $X_{i,j,0}, X_{i,j,1}, X_{i,j,2}, \dots$ are randomly chosen from $\{0,1,2, \dots, p-1\}$, which we may often identify as $\bF_p$, the finite field of $p$ elements.  The most natural example is when each $X_{i,j,l}$ is distributed uniformly at random, which is equivalent to saying that $X_{ij}$ is given by the Haar measure on $\bZ_p$. In \cite{FW}, Friedman and Washington computed the distribution of $\cok(X)$ of a random matrix $X \in \M_n(\bZ_p)$ whose $n^2$ independent entries $(X_{ij})_{1 \leq i, j \leq n}$ are Haar-random in $\bZ_p$. More specifically, \cite[Proposition 1]{FW} says
\begin{equation}\label{FW}
\underset{X \in \M_n(\bZ_p)^{\Ha}}{\Prob}(\cok(X) \simeq G) = \frac{1}{|\Aut(G)|}\prod_{i=1}^{n}(1 - p^{i})\prod_{j=n-r_p(G)+1}^{n}(1 - p^{-j}),
\end{equation}

as long as $n \geq r_p(G) := \dim_{\bF_p}(G/pG)$ (which otherwise gives $0$ for the probability), where $\Aut(G)$ is the automorphism group of $G$.

\begin{rmk} We shall always assume that $\M_n(\bZ_p)$ has the Borel $\sg$-algebra or the discrete $\sg$-algebra. We have used the notation $\M_n(\bZ_p)^{\Ha}$ above to indicate that each independent entry $X_{ij}$ of a random matrix $X \in \M_n(\bZ_p)^{\Ha}$ is Haar-random, which also assumes that we are using the Borel $\sg$-algebra.
\end{rmk}

\hspace{3mm} In \cite{Woo19}, Wood showed that as long as the first digit $X_{i,j,0}$ of each independent random variable $X_{ij}$ is not too concentrated on a single value in \eqref{digits}, when $n \ra \infty$, the distribution of the cokernel in \eqref{FW} is insensitive to which measure we choose on $\M_n(\bZ_p)$. More specifically, \cite[Theorem 1.2]{Woo19} says:

\begin{thm}[Wood]\label{Woo} Let $0 < \ep < 1$ be a real number, and fix a finite abelian $p$-group $G$. For each $n \in \bZ_{\geq 1}$, suppose that $\M_n(\bZ_p) = \bZ_{p}^{n^2}$ is equipped with a probability measure, where each random $X \in \M_n(\bZ_p)$ has $n^2$ independent entries, each $X_{ij}$ of which satisfies
\[\max_{a \in \bF_p}\lt(\underset{X_{ij} \in \bZ_p}{\Prob}(X_{i,j,0} = a)\rt) \leq 1 - \ep\]
in terms of the notation \eqref{digits}. Then
\[\lim_{n \ra \infty}\underset{X \in \M_n(\bZ_p)}{\Prob}(\cok(X) \simeq G) = \frac{1}{\abs{\Aut(G)}}\prod_{i=1}^{\infty}(1 - p^{-i}).\]
\end{thm}

\hspace{3mm} Theorem \ref{Woo} is extremely surprising in comparison to \eqref{FW} because each independent entry $X_{ij}$ is now allowed to be given the Haar measure by setting $\ep = 1 - 1/p$ or a probability measure far from the Haar measure such as the $(0,1)$-Bernoulli measure, where $X_{ij}$ takes the value of $0$ or $1$ with probability $1/2$ each, by setting $\ep = 1/2$. This is called a \textbf{universality result} because as $n \ra \infty$, multiple ways to choose measures on $X_{ij}$ do not change the result. This naturally brings the question about how much relaxation we can allow for each independent entry $X_{ij}$. 

\begin{rmk}
The right-hand side of the conclusion of Theorem \ref{Woo} defines a discrete probability distribution on the set of isomorphism classes of finite abelian $p$-groups called the \textbf{Cohen--Lentra distribution}. See \cite[\S 1]{Woo19} or \cite[\S 1]{CH} for its number-theoretic origination.
\end{rmk}

\hspace{3mm} In this paper, we investigate the complementary situation where each independent entry $X_{ij}$ of a random matrix $X$ is concentrated on a single residue modulo $p$ (i.e., $X_{i,j,0}$ is constant). This may look pathological at first. In \cite[p.384]{Woo19}, Wood notes that ``some condition that the matrix entries are not too concentrated, like [$\max_{r \in \bF_p}({\Prob}_{X_{ij} \in \bZ_p}(X_{i,j,0} = r)) \leq 1 - \ep$], is certainly necessary, since if the matrices had even two rows whose values were all $r \pmod{p}$, then [$\cok(X)$] could never be the trivial group.'' Indeed, there needs to be some condition to be imposed in order to avoid a trivial conclusion. That is, to satisfy $\cok(X) \simeq G$ for a finite abelian $p$-group $G$, we must have $\cok(\bar{X}) \simeq G/pG$, where $\bar{X}$ is the residue of $X$ modulo $p$.

\hspace{3mm} In fact, although they did not explicitly mention this, Friedman and Washington \cite[p.235]{FW} showed that there is an interesting behavior for a matrix $X \in \M_n(\bZ_p)$ with constant $X_{i,j,0}$ and uniform $X_{i,j,1}, X_{i,j,2}, \dots$. We state their result as follows:

\begin{thm}[Friedman and Washington]\label{FW2} Fix a finite abelian $p$-group $G$ and $n \in \bZ_{\geq 1}$. For any $A_n \in \M_n(\bF_p)$ such that $\cok(A_n) \simeq G/pG$, we have the following conditional probability:
\[\underset{X \in \M_n(\bZ_p)^{\Ha}}{\Prob}(\cok(X) \simeq G \mid X \equiv A_n \ppmod{p}) = \frac{p^{r_p(G)^2}\prod_{i=1}^{r}(1 - p^{-i})^2}{|\Aut(G)|},\]
where $r_p(G) := \dim_{\bF_p}(G/pG)$.
\end{thm}

\hspace{3mm} Theorem \ref{FW2} was first generalized by the authors in \cite{CH}, by the first author and Kaplan in \cite[Theorem 1.6]{CK}, and then by the first author, Liang, and Strand \cite[Theorem 1.3]{CLS}, all of which are special cases of the following conjecture from \cite[Conjecture 1.8]{CK}. For any commutative ring $R$, given any $R$-modules $G$ and $H$, we write $G \simeq_R H$ to mean that $G$ and $H$ are isomorphic as $R$-modules. We write $G \simeq H$ to mean $G \simeq_{\bZ} H$. We denote by $\Aut_R(G)$ the group of $R$-linear automorphisms of $G$, with which we note that $\Aut_{\bZ}(G) = \Aut(G)$. 

\begin{conj}[Cheong and Kaplan]\label{CK} Let $P(t) \in \bZ_p[t]$ be a non-constant monic square-free polynomial so that we may write $P(t) = P_1(t) \cdots P_l(t)$, where $P_j(t) \in \bZ_p[t]$ are monic polynomials whose reductions $\bar{P}_j(t)$ modulo $p$ are distinct and irreducible in $\bF_p[t]$. For any finite-sized $\bZ_{p}[t]/(P(t))$-module $G$ and $A_n \in \M_n(\bF_p)$ such that $\cok(\bar{P}(A_n)) \simeq_{\bF_p[t]} G/pG$, we must have
\[\underset{X \in \M_n(\bZ_p)^{\Ha}}{\Prob}(\cok(P(X)) \simeq_{\bZ_p[t]} G \mid X \equiv A_n \ppmod{p}) = \frac{1}{|\Aut_{\bZ_p[t]}(G)|}\prod_{j=1}^{l} p^{r_{q_j}(G)^2}\prod_{i=1}^{r_{q_j}(G)}(1 - q_j^{-i})^2,\]
where we wrote $q_j := p^{\deg(P_j)}$ with $\bF_{q_j} := \bF_p[t]/(\bar{P}_j(t))$, the finite field of $q_j$ elements, and $r_{q_j}(G) := \dim_{\bF_{q_j}}(G/pG \otimes_{\bF_p[t]} \bF_{q_j})$.
\end{conj}

\hspace{3mm} Note that $\cok(P(X)) = \bZ_p^n/P(X)\bZ_p^n$ has a $\bZ_p[t]/(P(t))$-module structure, whose action of $t$ is given by the left multiplication by $X$, so assuming that $G$ is a module over $\bZ_p[t]/(P(t))$ in Conjecture \ref{CK} is inevitable. A special case of our main theorem is the following:

\begin{thm} Conjecture \ref{CK} is true.
\end{thm}

\hspace{3mm} Our main theorem is more general than the above statement. Namely, we are able to compute the probability in the conclusion of Conjecture \ref{CK} for any monic $P(t) \in \bZ_p[t]$ without any square-free condition on its reduction $\bar{P}(t) \in \bF_p[t]$ modulo $p$. We fix a non-constant monic $P(t) \in \bZ_p[t]$ and consider the unique factorization
\begin{equation}\label{fac}
\bar{P}(t) = \bar{P}_1(t)^{m_1} \bar{P}_2(t)^{m_2} \cdots \bar{P}_l(t)^{m_l},
\end{equation}
where $\bar{P}_1(t), \bar{P}_2(t), \dots, \bar{P}_l(t)$ are distinct monic irreducible polynomials in $\bF_p[t]$ and $m_1, m_2, \dots, m_l \in \bZ_{\geq 1}$. We shall also write $d_j := \deg(\bar P_j(t))$. Given an $\bF_p[t]/(P(t))$-module $M$, we write
\[u_j(M) := \dim_{\bF_{p^{d_j}}}\lt(\bar{P}_j(t)^{m_j - 1} M_j \rt),\]

where $M_j := M  \ot_{\bF_p[t]/(\bar{P}(t))} \bF_p[t]/(\bar{P}_j(t)^{m_j})$.

\hspace{3mm} We are now ready to state one of our main theorems:
\begin{thm}\label{main} Let $n \in \bZ_{\geq 1}$. Fix a finite-sized $\bZ_p[t]/(P(t))$-module $G$ and $A_n \in \M_n(\bF_p)$ such that $\cok(\bar P(A_n)) \simeq_{\bF_p[t]} G/pG$. If $G$ satisfies
\[|\Hom_{\bZ_{p}[t]}(G, \bF_{p^{d_{j}}})| = |\Ext_{\bZ_{p}[t]/(P(t))}^{1}(G, \bF_{p^{d_{j}}})|\] 

for $1 \leq j \leq l$, then
\[\underset{X \in \M_n(\bZ_p)^{\Ha}}{\Prob}(\cok(P(X)) \simeq_{\bZ_p[t]} G \mid X \equiv A_n \ppmod{p}) = \frac{|\Aut_{\bZ_p[t]}(G/pG)|\prod_{j=1}^{l}\prod_{i=1}^{u_j(G/pG)}(1 - p^{-id_j})}{|\Aut_{\bZ_p[t]}(G)|}.\]

Otherwise, the probability is $0$.
\end{thm}

\hspace{3mm} In Theorem \ref{main}, we note that having $\cok(P(A_n)) \simeq_{\bF_p[t]} G/pG$ guarantees that there exists $g \in \GL_n(\bF_p)$ such that
\[A_n = g\begin{bmatrix}
J & \ast \\
0 & J'
\end{bmatrix}g^{-1}\]
in $\M_n(\bF_p)$, where $J \in \M_{n-r}(\bF_p)$ and $J' \in \M_r(\bF_p)$ with $r  = r_p(G)$ such that every eigenvalue of $J$ in $\ol{\bF_p}$ is not a root of $P(t)$, while every eigenvalue of $J'$ in $\ol{\bF_p}$ is a root of $P(t)$. Moreover, we have
\[\cok(P(A_n)) \simeq \cok\lt( P\lt(g\begin{bmatrix}
J & \ast \\
0 & J'
\end{bmatrix}g^{-1}\rt)\rt) = \cok\lt( gP\lt(\begin{bmatrix}
J & \ast \\
0 & J'
\end{bmatrix}\rt)g^{-1}\rt),\]
and for any lift $\tilde{g} \in \GL_n(\bZ_p)$ of $g$, the conjugation by $\tilde{g}$ preserves the Haar measure on $\M_n(\bZ_p)$. Thus, Theorem \ref{main} is equally strong, even if we assume that 
\begin{equation}\label{specialA}
A_n = \begin{bmatrix}
J & \ast \\
0 & J'
\end{bmatrix}
\end{equation}
with $J$ and $J'$ as above. (Most importantly, we recall that every eigenvalue of $J \in \M_{n-r}(\bF_p)$ is not a root of $P(t)$ and $r = r_p(G)$.) For this specific form of $A_n$, Theorem \ref{main} holds in a more general setting, which can be seen as a universality result:

\begin{thm}\label{main2} Let $n \in \bZ_{\geq 1}$. Fix a finite-sized $\bZ_p[t]/(P(t))$-module $G$ and $A_n \in \M_n(\bF_p)$ such that $\cok(\bar{P}(A_n)) \simeq_{\bF_p[t]} G/pG$. Suppose that $A_n$ is of the form \eqref{specialA}, and consider any probability measure on $\M_n(\bZ_p)$ such that all entries of $X$ are independent and the entries in the bottom-right $r \times r$ submatrix of $X$ follow the Haar measure. If $G$ satisfies
\[|\Hom_{\bZ_{p}[t]}(G, \bF_{p^{d_{j}}})| = |\Ext_{\bZ_{p}[t]/(P(t))}^{1}(G, \bF_{p^{d_{j}}})|\] 

for $1 \leq j \leq l$, then
\[\underset{X \in \M_n(\bZ_p)}{\Prob}(\cok(P(X)) \simeq_{\bZ_p[t]} G \mid X \equiv A_n \ppmod{p}) = \frac{|\Aut_{\bZ_p[t]}(G/pG)|\prod_{j=1}^{l}\prod_{i=1}^{u_j(G/pG)}(1 - p^{-id_j})}{|\Aut_{\bZ_p[t]}(G)|}.\]

Otherwise, the probability is $0$.
\end{thm}

\begin{rmk} When $P(t)$ is square-free modulo $p$ (i.e., $m_1 = m_2 = \cdots = m_l = 1$ in \eqref{fac}), the condition 
\[|\Hom_{\bZ_{p}[t]}(G, \bF_{p^{d_{j}}})| = |\Ext_{\bZ_{p}[t]/(P(t))}^{1}(G, \bF_{p^{d_{j}}})|,\] 
is always satisfied for all $1 \leq j \leq l$ by \cite[Lemma 2.2]{CY}. This is why in Conjecture \ref{CK} such conditions were not visible. The following proposition explains more about what happens in general:
\end{rmk}

\begin{prop}\label{prep} Let $n \in \bZ_{\geq 1}$. Fix a finite-sized module $G$ over $\bZ_{p}[t]/(P(t))$ and $A_n \in \M_n(\bF_p)$ such that $\cok(P(A_n)) \simeq_{\bF_p[t]} G/pG$. Then the following are equivalent:
\be
	\item There exists $X \in \M_n(\bZ_p)$ such that $\cok(P(X)) \simeq_{\bZ_{p}[t]} G$ and $X \equiv A_n \ppmod{p}$.
	\item We have $|\Hom_{\bZ_{p}[t]}(G, \bF_{p^{d_{j}}})| = |\Ext_{\bZ_{p}[t]/(P(t))}^{1}(G, \bF_{p^{d_{j}}})|$ for $1 \leq j \leq l$.
\ee
\end{prop}

\hspace{3mm} Theorem \ref{main} implies the Haar measure case of the following theorem of the first author and Yu, whose special case (with Haar measure, assuming $\bar{P}(t) \in \bF_p[t]$ is square-free) was first proved by Lee \cite{LeeA}:

\begin{thm}[Cheong--Yu]\label{CY} Let $0 < \ep < 1$ be a real number, and fix a finite-sized module $G$ over $\bZ_{p}[t]/(P(t))$. For each $n \in \bZ_{\geq 1}$, suppose that $\M_n(\bZ_p) = \bZ_{p}^{n^2}$ is equipped with a probability measure, where each random $X \in \M_n(\bZ_p)$ has $n^2$ independent entries, each $X_{ij}$ of which satisfies
\[\max_{a \in \bF_p}\lt(\underset{X_{ij} \in \bZ_p}{\Prob}(X_{i,j,0} = a)\rt) \leq 1 - \ep,\]
in terms of the notation \eqref{digits}. If $G$ satisfies
\[|\Hom_{\bZ_{p}[t]}(G, \bF_{p^{d_{j}}})| = |\Ext_{\bZ_{p}[t]/(P(t))}^{1}(G, \bF_{p^{d_{j}}})|\] 

for $1 \leq j \leq l$, then

\[\lim_{n \ra \infty}\underset{X \in \M_{n}(\bZ_{p})}{\Prob}(\cok(P(X)) \simeq_{\bZ_{p}[t]} G)
= \dfrac{1}{|\Aut_{\bZ_{p}[t]}(G)|}\displaystyle\prod_{j=1}^{l}\displaystyle\prod_{i=1}^{\infty}\lt(1 - p^{-id_{j}}\rt).\]
Otherwise the limit is $0$.
\end{thm}

\begin{rmk} It turns out that random matrices $X$ with concentrated residue $A_n$ gives many constraints on the entries, and essentially, Theorem \ref{main2} is the best possible result one may hope for their universality. For example, consider the case $P(t) = t$ and $A_n = \diag(1,1, \dots, 1, 0)$, the $n \times n$ diagonal entries with $(0,1)$-diagonal entries with one $0$ entry. If we consider $X = A_n + pB$ with $B \in \M_n(\bZ_p)$, then for any odd $p$, if the $(n,n)$-entry of $B$ never takes $0$, then the conclusion of Theorem \ref{main2} does not hold. (More examples and counterexamples can be made from the arguments used in the proof of Theorem \ref{main2}, which is at the end of this paper.)
\end{rmk}

\subsection{Relevance to past and future works} The first special case of Theorem \ref{main} with $P(t) = t$ was shown by Friedman and Washington, as stated in Theorem \ref{FW2}. When $P(t)$ is square-free modulo $p$, Theorem \ref{main} was partially proven by the authors \cite[Lemma 5.2]{CH}, the first author and Kaplan \cite[Theorem 1.6]{CK} for $d_1, \dots, d_l \leq 2$, and the first author, Liang, and Strand \cite[Theorem 1.3]{CLS} for $l = 1$. Assuming that $P(t)$ is square-free modulo $p$ makes the problem more accessible because then the ring $\bZ_p[t]/(P(t))$ is a finite product of DVRs, and one of our contributions is to get around this difficulty for a general monic polynomial $P(t) \in \bZ_p[t]$, where the ring $\bZ_p[t]/(P(t))$ is much more complicated.

\hspace{3mm} The first universality result for random integral matrices appears in Wood's breakthrough \cite[Theorem 1.3]{Woo17} for symmetric $\bZ_p$-matrices, which generalizes its Haar measure version proven by Clancy, Kaplan, Leake, Payne, and Wood \cite[Theorem 2, summing over all the parings]{CKLPW}. Ever since, her techniques have been used to extend many results about about Haar-random $\bZ_p$-matrices to random $\bZ_p$-matrices each of whose independent entry is not too concentrated on a single residue modulo $p$ (i.e., $X_{i,j,0}$ in \eqref{digits} is not too concentrated on a single value). For example, universality results from \cite{CY}, \cite{NV}, \cite{Woo17}, and \cite{Woo19} generalize Haar measusre results from \cite{LeeA}, \cite{Van}, \cite{CKLPW}, and \cite{FW}, respectively.

\hspace{3mm} Several authors \cite{FW, CH, CK, CLS} have studied properties of random $X \in \M_n(\bZ_p)$ when $X_{i,j,0}$ is constant, but all the other $p$-adic digits $X_{i,j.1}, X_{i,j,2}$, and so on in \eqref{digits} are given the uniform distribution. Theorem \ref{main2} provides the first universality result with $X_{i,j,0}$ being constant as it allows us to choose any distributions for all the other $p$-adic digits, as long as $A_n$ has a specific form in \eqref{specialA} and the bottom-right $r_p(G) \times r_p(G)$ submatrix of $X$ follows the Haar measure. This seems to be the best universality result that we may hope for in this concentrated residue setting.

\hspace{3mm} Our work opens up numerous questions about the behavior of random integral matrices with fixed residue. To begin with, we may ask about analogues of Theorems \ref{main} and \ref{main2} for different random matrix models such as symmetric matrices or skew-symmetric matrices. We may ask about the concentrated residue version for \cite{NV}, which deals with the cokernel of product of $\bZ_p$-matrices. We may ask about the concentrated residue version for \cite{LeeB}, which deals with the cokernel of Hermitian matrices over a quadratic extension of $\bZ_p$.

\subsection{Methodology and brief outline of the paper} The majority of the work is going into proving Theorem \ref{main}. We go through a series of reductions from \S \ref{HaStart} to \S \ref{HaEnd} for this. We shall see that behind this, there is an interesting equidistibution result (Theorem \ref{equi}) for matrices over $\bZ_p[t]/(P(t))$, which we eventually prove by establishing a noncommutative version of the Weierstrass preparation theorem for the matrix ring $\M_n(\bZ_p)$ (Theorems \ref{thm:weierstrass} and \ref{thm:finweierstrass}). Then to prove Theorem \ref{main2}, we use the strategy to compute the moments (discussed in \S \ref{red}) of the distribution of $\cok(P(X))$ to determine the distribution. One of the major difficulties in our work in comparison to previous works is that each moment of our distribution cannot be explicitly written. We deal with this difficulty by using Theorem \ref{main}, to observe (in \S \ref{ind}) to get a candidate for the moment $M_H$ only depending on a fixed module $H$ over a suitable ring.

\section{Proof of Theorem \ref{main} from an equidistribution result} \label{HaStart}

\hspace{3mm} From this section to \S \ref{HaEnd}, we prove Proposition \ref{prep} and Theorem \ref{main}. Given any $A_n \in \M_n(\bF_p)$, we shall write
\[\M_n(\bZ_p)_{A_n} := \{X \in \M_n(\bZ_p) : X \equiv A_n \ppmod{p}\}\]
so that 
\[\underset{X \in \M_n(\bZ_p)}{\Prob}(\cok(P(X)) \simeq_{\bZ_p[t]} G \mid X \equiv A_n \ppmod{p}) = \underset{X \in \M_n(\bZ_p)_{A_n}}{\Prob}(\cok(P(X)) \simeq_{\bZ_p[t]} G).\]
That is, we consider $\M_n(\bZ_p)_{A_n}$ as the sample space instead of mentioning conditional probabilities for the statement of Theorem \ref{main}. The \textbf{Haar measure} on $\M_n(\bZ_p)_{A_n}$ is defined to be the probability measure induced by the Haar measure of $\M_n(\bZ_p)$.

\begin{rmk} In this section, all probability measures we deal with are the Haar measures. For example, we assume $\M_n(\bZ_p)_{A_n} = \M_n(\bZ_p)_{A_n}^{\Ha}$. We shall keep this assumption till \S \ref{HaEnd}. Starting from Section \ref{red}, we shall drop this assumption.
\end{rmk}

\subsection{Linearization and equidistribution} For any $X \in \M_n(\bZ_p)$, we note that
\begin{equation}\label{Lee}
\cok(P(X)) \simeq_{R} \cok_{R}(X - \bar{t}I_n) := \frac{R^n}{((X - \bar{t}I_n)R^n)},
\end{equation}
where 
\bi
	\item $I_n$ is the $n \times n$ identity matrix,
	\item $R := \bZ_p[t]/(P(t))$, and
	\item $\bar{t} \in R$ is the image of $t$.
\ei	
We call this isomorphism \textbf{Lee's linearization trick}, first used in \cite{LeeA}. The isomorphism linearizes our problem by shifting the difficulty of taking the polynomial push-forward $P(X)$ of $X$ into dealing with a more complicated ring $R$ instead of $\bZ_p$. This will be used not only for proving Theorem \ref{main} but also for proving Theorem \ref{main2} by using the version of \eqref{Lee} with
\bi
	\item $X \in \M_n(\bZ/p^k\bZ)$ for a given $k \in \bZ_{\geq 1}$,
	\item $P(t) \in (\bZ/p^k\bZ)[t]$ monic, and
	\item $R = (\bZ/p^k\bZ)[t]/(P(t))$
\ei
instead.

\hspace{3mm} The following is the linearized version of Proposition \ref{prep}.

\begin{prop}\label{prep'} Let $n \in \bZ_{\geq 1}$. Fix a finite size module $G$ over $R$ and $J_n \in \M_n(R/pR)$ such that $\cok(J_n) \simeq_{\bF_p[t]} G/pG$. Then the following are equivalent:
\be
	\item There exists $Z \in \M_n(R)$ such that $\cok(Z) \simeq_{R} G$ and $Z \equiv J_n \ppmod{p}$.
	\item We have $|\Hom_{\bZ_{p}[t]}(G, \bF_{p^{d_{j}}})| = |\Ext_{R}^{1}(G, \bF_{p^{d_{j}}})|$ for $1 \leq j \leq l$.
\ee
\end{prop}

\hspace{3mm} The following is the linearized version of Theorem \ref{main}. Shortly, we show that this version together with an equidistribution theorem implies Theorem \ref{main}. We let $R := \bZ_p[t]/(P(t))$ for the rest of this section.

\begin{thm}\label{Ha'} Keeping the hypotheses and notation in Proposition \ref{prep'}, if $G$ satisfies
\[|\Hom_{\bZ_{p}[t]}(G, \bF_{p^{d_{j}}})| = |\Ext_{R}^{1}(G, \bF_{p^{d_{j}}})|\]
for $1 \leq j \leq l$, then
\[\underset{Z \in \M_n(R)}{\Prob}(\cok(Z) \simeq_{\bZ_p[t]} G | Z \equiv J_n \ppmod{p}) = \frac{|\Aut_{\bZ_p[t]}(G/pG)|\prod_{j=1}^{l}\prod_{i=1}^{u_j(G/pG)}(1 - p^{-id_j})}{|\Aut_{\bZ_p[t]}(G)|}\]
for any $n \in \bZ_{\geq 1}$. Otherwise, the probability is $0$.
\end{thm}

\hspace{3mm} The key in deducing Proposition \ref{prep} and Theorem \ref{main} from Proposition \ref{prep'} and Theorem \ref{Ha'} is to establish the following surprising equidistribution result in its own right, a special case of which was first found by the first author, Liang, and Strand in \cite[Lemma 3.7]{CLS}. Write $d := \deg(P)$ for convenience from now on.

\begin{thm}\label{equi} For any $n \in \bZ_{\geq 1}$ and a finite size $R$-module $G$. For any $pY_1, pY_2, \dots, pY_{d-1} \in p\M_n(\bZ_p)$, we have
\[\underset{X \in \M_n(\bZ_p)_{A_{n}}}\Prob(\cok_{R}(X - \bar{t} I_n) \simeq_{R} G) = \underset{X \in \M_n(\bZ_p)_{A_{n}}}\Prob(\cok_{R}(X + \bar{t}(pY_1 - I_n) + \bar{t}^2pY_2 + \cdots + \bar{t}^{d-1}pY_{d-1}) \simeq_{R} G).\]
\end{thm}

\hspace{3mm} We now assume Theorems \ref{Ha'} and \ref{equi} and then show the purported implications:

\begin{proof}[Theorems \ref{Ha'} and \ref{equi} imply Theorem \ref{main}] Assume the hypotheses of Theorem \ref{main}. Let
\[\M_n(R)_{A_n - \bar{t}I_n} := \{Z \in \M_n(R) : Z \equiv A_n - \bar{t}I_n \ppmod{p}\}.\]
By Theorem \ref{equi} with $J_n = A_n - \bar{t}I_n$, we have
\begin{align*}
&\underset{Z \in \M_n(R)_{A_n - \bar{t}I_n}}\Prob(\cok_{R}(Z) \simeq_{\bZ_p[t]} G) \\
&= \int_{(X, pY_1, \dots, pY_{d-1}) \in \M_n(\bZ_p)_{A_n} \times (p\M_n(\bZ_p))^{d-1}}\bb{1}(\cok_{R}(X+\bar{t}(pY_1- I_n) + \bar{t}^2pY_2 + \cdots + \bar{t}^{d-1}pY_{d-1}) \simeq_{\bZ_p[t]} G) d(\rho_n \times \mu_n^{d-1}) \\
&= \int_{(pY_1, \dots, pY_{d-1}) \in (p\M_n(\bZ_p))^{d-1}}\underset{X \in \M_n(\bZ_p)_{A_n}}\Prob(\cok_{R}(X+\bar{t}(pY_1- I_n) + \bar{t}^2pY_2 + \cdots + \bar{t}^{d-1}pY_{d-1}) \simeq_{\bZ_p[t]} G) d \mu_n^{d-1} \\
&= \int_{(pY_1, \dots, pY_{d-1}) \in (p\M_n(\bZ_p))^{d-1}}\underset{X \in \M_n(\bZ_p)_{A_n}}\Prob(\cok_{R}(X-\bar{t} I_n) \simeq_{\bZ_p[t]} G) d \mu_n^{d-1} \\
&= \underset{X \in \M_n(\bZ_p)_{A_n}}\Prob(\cok_{R}(X-\bar{t} I_n) \simeq_{\bZ_p[t]} G) \\
&= \underset{X \in \M_n(\bZ_p)_{A_n}}\Prob(\cok(P(X)) \simeq_{\bZ_p[t]} G),
\end{align*}
where $\mu_n$ is the Haar measure of $p\M_n(\bZ_p)$ and $\rho_n$ is the Haar measure of $\M_n(\bZ_p)_{A_n}$, which is introduced right after Theorem \ref{main}. (We used Lee's linearization trick \eqref{Lee} at the end.) Hence, Theorems \ref{Ha'} and \ref{equi} imply Theorem \ref{main}.
\end{proof}

\begin{proof}[Proposition \ref{prep'} implies Proposition \ref{prep} assuming Theorems \ref{Ha'} and \ref{equi}] Let $G$ be a finite-sized $R$-module and $A_n \in \M_n(\bF_p)$ such that 
\[\cok(P(A_n)) \simeq_{\bF_p[t]} G/pG.\]
Let $J_n := A_n - \bar{t}I_n \in \M_n(R/pR)$. First, assume (1) of Proposition \ref{prep}: there exists $X \in \M_n(\bZ_p)$ such that $\cok(P(X)) \simeq_{R} G$ and $X \equiv A_n \ppmod{p}$. Then, we take $Z := X - \bar{t}I_n \in \M_n(R)$, which satisfies (1) of Proposition \ref{prep'} due to Lee's linearization trick \eqref{Lee}. This implies (2) of Propositions \ref{prep'} and \ref{prep}.

\hspace{3mm} Conversely, assume (2) of Proposition \ref{prep} (which is identical to Proposition \ref{prep'}). Then by (1) of Proposition \ref{prep'}, we have $Z \in \M_n(R)$ such that $Z \equiv A_n - \bar{t}I_n \ppmod{p}$. This implies that $Z = A_n + pY_0 + \bar{t}(pY_1 - I_n) + \bar{t}^2pY_2 + \cdots + \bar{t}^{d-1}pY_{d-1}$ for some $pY_0, pY_1, pY_2, \dots, pY_{d-1} \in p\M_n(\bZ_p)$. Take $X := A_n + pY_0 \in \M_n(\bZ_p)$, which satisfies $X \equiv A_n \ppmod{p}$. By Theorem \ref{equi}, the same argument as in the previous proof gives us
\[\underset{X \in \M_n(\bZ_p)_{A_n}}\Prob(\cok(P(X)) \simeq_R G) = \underset{X \in \M_n(\bZ_p)_{A_n}}\Prob(\cok_R(X - \bar{t}I_n) \simeq_R G) = \underset{Z \in \M_n(R)}\Prob(\cok_R(Z) \simeq_R G \mid Z \equiv J_n \ppmod{p}).\]
Then by Theorem \ref{Ha'}, the last probability is not $0$, so this implies (1) of Proposition \ref{prep}, as desired.
\end{proof}

\section{Proofs of Proposition \ref{prep'} and Theorem \ref{Ha'}}

\hspace{3mm} Recall that our current goal (from \S \ref{HaStart} to \S \ref{HaEnd}) is to prove Proposition \ref{prep} and Theorem \ref{main}. From the previous section, we know that in order to prove the desired statements, it suffices to prove Proposition \ref{prep'}, Theorem \ref{Ha'}, and Theorem \ref{equi}. In this section, we prove the first two of these.

\hspace{3mm} Since $R=\Zp[t]/(P(t))$ is not necessarily a PID, the proofs of Proposition \ref{prep'} and Theorem \ref{Ha'} differ significantly from the proof given in Friedman and Washington \cite{FW} (which corresponds to the case $P(t) = t$) due to the lack of the Smith normal form over $R$. Instead, we shall first develop a few formulas applicable to local Noetherian rings in general. They involve minimal resolutions, which we recall next.

\subsection{Minimal resolutions} Throughout this subsection, let $(R,\m, \ka)$ be a Noetherian local ring with maximal ideal $\m$ and residue field $\ka$. In addition, let $M$ be a finitely generated $R$-module. A \textbf{minimal resolution} of a finitely generated $R$-module $G$ is an exact sequence
\begin{equation}\label{eq:min_resol}
    \cdots \map[A_2] R^{b_1} \map[A_1] R^{b_0} \map[A_0] G \map[A_{-1}] 0
\end{equation}
such that the following equivalent\footnote{This equivalence can be deduced from Nakayama's lemma. (For example, it directly follows from \cite[Lemma 19.4]{E}.)} conditions hold:
\begin{enumerate}
    \item Each matrix $A_i$ with $i \geq 1$ has entries in $\m$;
    \item For each $i\geq 0$, we have that $b_i$ is the minimal number of generators for $\ker(A_{i-1}) = \im(A_{i})$.
\end{enumerate}
By (1), we have 
\begin{equation}\label{eq:resol_rank}
    b_i=\dim_{\ka}( \Tor^R_i(G,\ka) ) = \dim_{\ka} ( \Ext_R^i(G,\ka) ).
\end{equation}
In particular, $b_i$ only depends on $G$, but not on the resolution. Hence, we may write $\beta_i^R(G):=b_i$ and call it the $i$-th \textbf{Betti number} of $G$. We repetitively use that $\beta_0^R(G)=\dim_{\ka} (G/\m G)$ is the minimal number of generators of $G$, which is called the \textbf{rank} of $G$.

\hspace{3mm} We are ready to state the key formula we need in the proofs of Proposition \ref{prep'} and Theorem \ref{Ha'}. For our purpose, we only need the square-matrix case $u=0$ of the following theorem, but we present the general case because it does not appear to be in the literature. Given $m, n \in \bZ_{\geq 1}$, we denote by $\M_{n \times m}(A)$ the set of $n \times m$ matrices over a given ring $A$.

\begin{thm}\label{thm:fw_general} Let $(R,\m,\Fq)$ be a complete Noetherian local ring with a finite residue field $\bF_q$ of $q$ elements, and fix $u\in \bZ_{\geq 0}$. Let $G$ be a finite-sized $R$-module with Betti numbers $\beta^R_i(G) = b_i$. Then there exists $X \in \M_{n\times (n+u)}(R)$ with $\cok(X) \simeq_R G$ if and only if $n\geq b_0\geq b_1-u$. Moreover, with respect to the Haar measure, we have
    \begin{equation}\label{eq:fw_general}
        \underset{X \in \M_{n\times (n+u)}(R)}\Prob(\cok(X) \simeq_R G) = \frac{1}{\abs{\Aut_R (G)}\, \abs{G}^u} \prod_{i=u+b_0-b_1+1}^{n+u}(1-q^{-i}) \prod_{j=n-b_0+1}^n (1-q^{-j})
    \end{equation}
    if $n\geq b_0\geq b_1-u$, and zero otherwise. 
\end{thm}

\hspace{3mm} We defer the proof of Theorem \ref{thm:fw_general} to \S \ref{sec:fw_general_proof}. 

\subsection{Fixing a residue class} Proposition \ref{prep'} and Theorem \ref{Ha'} concern Haar-random matrices with concentrated residue class, but Theorem \ref{thm:fw_general} is just about Haar-random matrices. In order to apply Theorem \ref{thm:fw_general}, we need the following lemma, whose DVR case was implicitly noted in \cite{FW}:

\begin{lem}
Fix $m, n \in \bZ_{\geq 1}$. Let $(R,\m,\Fq)$ be a complete Noetherian local ring with a finite residue field $\Fq$ of $q$ elements equipped with the Haar measure, and let $\fa \sub \m$ be an ideal of $R$ with $R/\fa$ of finite size. Let $G$ be a finite-length $R$-module. Consider any $\bar X\in \M_{n\times m}(R/\fa)$ satisfying $\cok_{R/\fa}(\bar X)\simeq_{R} G/\fa G$. Then the conditional probability
    \begin{equation}\label{eq:fw_fiberwise}
        \underset{X\in \M_{n\times m}(R)}\Prob\parens[\bigg]{\cok(X) \simeq_R G \,\vrule\, X\equiv \bar X \pmod \fa}
    \end{equation}
    does not depend on $\bar X$. 
\label{lem:equidistribution_residue}
\end{lem}

\hspace{3mm} We defer the proof of Lemma \ref{lem:equidistribution_residue} to \S \ref{sec:equidistribution_residue_proof}. Theorem \ref{thm:fw_general} and Lemma \ref{lem:equidistribution_residue} immediately imply the following theorem, which is used in the proofs of Proposition \ref{prep'} and Theorem \ref{Ha'}.

\begin{thm}\label{thm:fw_fiberwise}
    Let $(R,\m,\Fq)$ be a complete Noetherian local ring with a finite residue field $\Fq$ with $q$ elements equipped with the Haar measure, and let $\fa \sub \m$ be an ideal of $R$ with $R/\fa$ of finite size. Let $G$ be a finite-size $R$-module, and let $\bar X\in \M_{n\times (n+u)}(R/\fa)$ be such that $\cok_{R/\fa}(\bar X) \simeq_{R} G/\fa G$. Then for any $u\in \bZ_{\geq 0}$, we have
    \begin{equation}\label{eq:fw_fiberwise_formula}
        \underset{X\in \M_{n\times (n+u)}(R)}\Prob\parens[\bigg]{\cok(X) \simeq_R G \,\bigg|\, X\equiv \bar X \pmod \fa} = \begin{cases}
            {\displaystyle \frac{\abs{\Aut(G/\fa G)}}{\abs{\Aut_R(G)}\abs{\fa G}^u} \prod_{i=u+b_0-b_1+1}^{u+b_0-b_1'} (1-q^{-i})}, & b_0\geq b_1-u,\\
            0, & b_0<b_1-u,
        \end{cases} 
    \end{equation}
     where $b_i = \beta_i^R(G)$ for $i = 0, 1$ and $b_1' = \beta_1^{R/\fa}(G/\fa G)$. In particular, the conditional probability above does not depend on $n$.
\end{thm}

\begin{rmk} In the above theorem, we always have $b_1'\leq b_1$ (by Lemma \ref{lem:betti_wrt_quotient_ring} (2)). It is possible to have an empty product, which we consider as $1$ as usual. Furthermore, even though \eqref{eq:fw_fiberwise_formula} does not depend on $n$, the hypotheses of Theorem \ref{thm:fw_fiberwise} forces $n\geq b_0$. Indeed, we have $\beta_0^{R/\fa}(G/\fa G) = \beta_0^R(G) = b_0$ because both are equal to $\dim_{\Fq}( G/\m G )$. Therefore, the existence of $\bar X\in \M_{n\times (n+u)}(R/\fa)$ with $\cok(\bar X) \simeq_R G/\fa G$ implies $n\geq b_0$. 
\end{rmk}

\hspace{3mm} We use (1) and (2) of the following, and (3) will be used later:

\begin{lem}\label{lem:betti_wrt_quotient_ring}
    Let $(R,\m)$ be a Noetherian local ring. Suppose $\fa\sub \m$ is an ideal of $R$ and $G$ is a finitely generated $R$-module. Then we have
    \begin{enumerate}
        \item $\beta_0^{R/\fa}(G/\fa G)=\beta_0^R(G)$;
        \item $\beta_1^{R/\fa}(G/\fa G)\leq \beta_1^R(G)$;
        \item If we assume furthermore that $\fa=\m \fb$ for some ideal $\fb \sub R$, and $\fb G = 0$, then $\beta_1^{R/\fa}(G)= \beta_1^R(G)$.
    \end{enumerate}
\end{lem}

\begin{proof} Let $\ka = R/\m$, the residue field of $R$. Write $b_i = \beta_i^R(G)$ and $b_i' = \beta_i^{R/\fa}(G/\fa G)$ for $i = 0, 1$.
    ~
    \begin{enumerate}
        \item This follows because both sides are equal to $\dim_{\ka}(G/\m G)$.
        \item Let
        \begin{equation*}
            \dots \to R^{b_1} \to R^{b_0} \to G \to 0
        \end{equation*}
        be a minimal resolution of $M$ over $R$. Tensoring with $R/\fa$, we have an exact sequence
        \begin{equation*}
            (R/\fa)^{b_1} \to (R/\fa)^{b_0} \to G/\fa G \to 0.
        \end{equation*}
        
       Since $b_0 = b'_0$, by the definition of a minimal resolution of $G/\fa G$ over $R/\fa$, we have $b_1\geq b_1'$.
        \item Under the given hypotheses, we want to show $b'_1 = b_1$. Note that $\fb G=0$ implies $\fa G=0$, so $G$ is a finitely generated $R/\fa$-module. Using a minimal resolution of $G$ over $R/\fa$, we get a matrix $\bar X \in \Mat_{b_0\times b_1'}(R/\fa)$ such that $\cok_{R/\fa}(\bbar X) \simeq_R G$. Pick any lift $X\in \Mat_{b_0\times b_1'}(R)$ of $\bar X$, and let $M := \cok_R(X)$, then we have $M/\fa M \simeq_R M \otimes_R (R/\fa)\simeq_R G$. By Lemma \ref{lem:quotient_determines_original} (proven below), we must have $G \simeq_R M = \cok(X)$. In other words, there exists an exact sequence
        \begin{equation*}
            R^{b_1'} \map[X] R^{b_0} \to G \to 0.
        \end{equation*}
        
        By the definition of a minimal resolution of $M$ over $R$, we have $b_1'\geq b_1$. Combined with part (2), we get $b_1'=b_1$. \qedhere
    \end{enumerate}
\end{proof}

\begin{lem}\label{lem:quotient_determines_original} Let $(R,\m)$ be a Noetherian local ring. Fix an ideal $\fb \sub R$ and let $\fa := \m \fb$. If $G$ is a finitely generated $R/\fa$-module such that $\fb G = 0$, and $M$ is a finitely generated $R$-module such that $M/\fa M\simeq_R G$, then $\fa M = 0$ so that $M \simeq_R G$.
\end{lem}
\begin{proof} Since $\fb G=0$, we have $0=\fb(M/\fa M)=\fb M/\m \fb M$. By Nakayama's lemma, $\fb M = 0$, so $\fa M = 0$. Therefore, $G \simeq_R M/\fa M=M$.
\end{proof}

\begin{proof}[{Proof that Theorem \ref{thm:fw_general} and Lemma \ref{lem:equidistribution_residue} imply Theorem \ref{thm:fw_fiberwise}}]
Consider the set 
\[\M_{n,u,G}(R) := \{X\in \M_{n\times (n+u)}(R) : \cok(X) \simeq_R G\}.\] 

Similarly, consider the finite nonempty set 
 \[\M_{n,u,G/\fa G}(R/\fa) := \{X'\in \M_{n\times (n+u)}(R/\fa) : \cok(X') \simeq_R G/\fa G\}\]
    
We have a map $\Phi: \M_{n,u,M}(R)\to \M_{n\times (n+u)}(R/\fa)$ that sends $X$ to $(X \bmod \fa)$. By Lemma \ref{lem:equidistribution_residue}, the fibers of $\phi$ have constant measure which also implies that $\Phi$ is surjective. As a result, we have 
    \begin{equation*}
        \underset{X\in \M_{n\times (n+u)}(R)}\Prob \parens[\bigg]{\cok_R(X) \simeq_R G \text{ and } X\equiv \bar X \ppmod \fa} = \mu_{n \times (n+u)}(\Phi^{-1}(\bar X)) = \frac{\mu_{n \times (n + u)}(\M_{n,u,G}(R))}{\#\M_{n,u,G/\fa G}(R/\fa)},
    \end{equation*}
 where $\mu_{n \times (n+u)}$ is the Haar measure of $\M_{n \times (n+u)}(R)$.

\hspace{3mm} On the right-hand side, we apply Theorem \ref{thm:fw_general} for the $R$-module $G$ to the numerator and apply Theorem \ref{thm:fw_general} for the $R/\fa$-module $G/\fa G$ to the denominator. (Note that the ring $R/\fa$ and the module $G/\fa G$ satisfy the assumption of Theorem \ref{thm:fw_general}.) By Lemma \ref{lem:betti_wrt_quotient_ring} (2) $b_1'\leq b_1$, so the desired conditional probability then follows immediately.
\end{proof}

\hspace{3mm} We shall first show that Theorems \ref{thm:fw_general} and \ref{thm:fw_fiberwise} imply Proposition \ref{prep'} and Theorem \ref{Ha'}. Then we shall prove Theorem \ref{thm:fw_general} and Lemma \ref{lem:equidistribution_residue}.

\subsection{Some specifics about $\Zp[t]/(P(t))$} Throughout this subsection, assume $P(t)\in \Zp[t]$ is monic and the reduction of $P(t)$ modulo $p$ is of the form $\bar Q(t)^m$, where $m\geq 1$ and $\bar Q(t)$ is irreducible in $\Fp[t]$. In other words, we assume $l=1$ in \eqref{fac}. Then $R=\Zp[t]/(P(t))$ is a local ring\footnote{Given any maximal ideal $\mf{m}$ of $\bZ_p[t]/(P(t))$, we can show that $p \in \mf{m}$ by observing that $\mf{m}$ is finite over $\bZ_p$ and applying Nakayama's lemma. From here, it follows that the image of $Q(t)^m$ is in $\m$, so the image of $Q(t)$ must be in $\m$ so that $\mf{m} = (p, Q(t))/(P(t))$.} with maximal ideal $\m=(p,Q(t)) / (P(t))$, where $Q(t) \in \bZ_p[t]$ is any lift of $\bar Q(t)$, with the residue field $\Fp[t]/(\bar Q(t))$, a finite field of size $q:=p^{\deg \bar Q(t)}$.

\hspace{3mm} We shall apply Theorem \ref{thm:fw_fiberwise} with $\fa=pR$. The formula we get involves taking the first Betti number over the ring $R/\fa$. To explicitly compute it, we observe that $R/pR$ is a DVR quotient. Indeed, we may identify
\begin{equation*}
    \frac{R}{pR}=\frac{\Fp[t]}{(\bbar Q(t)^m)}=\frac{T}{(\pi^m)},
\end{equation*}
where $T$ is the $\bar Q(t)$-adic completion of $\Fp[t]$ and $\pi$ is the image of $\bar Q(t)$ in $T$. We note that that $T$ is a DVR with uniformizer $\pi$ and residue field $\Fq$. 

\begin{lem}\label{lem:betti_dvr_quotient}
    Let $(T, (\pi), \ka)$ be any DVR, and $m \in \bZ_{\geq 1}$. Let $G$ be a finite-length module over $T/(\pi^m)$. Then
    \begin{equation*}
        \beta_0^{T/(\pi^m)}(G) - \beta_1^{T/(\pi^m)}(G)=\dim_{\ka}(\pi^{m-1}(G)).
    \end{equation*}
\end{lem}
\begin{proof}
By the classification of finitely generated modules over $T/(\pi^m)$, it suffices to consider the case $G = T/(\pi^a)$ with $1\leq a\leq m$. The zeroth step of the minimal resolution of $G$ is given by the quotient map $T/(\pi^m) \tra G$, so $\beta_0^{T/(\pi^m)}(G) = 1$. If $a = m$, the quotient map $T/(\pi^m) \tra G$ is an isomoprhism, so $\beta_1^{T/(\pi^m)}(G) = 0$. In this case, we also have $\dim_{\ka}(\pi^{m-1}(G)) = \dim_{\ka}(\pi^{m-1}T/\pi^{m}T) = 1$. Otherwise, we have $a \leq m-1$. Then the kernel of the quotient map $T/(\pi^m) \tra G$ is minimally generated by one generator, so $\beta_1^{T/(\pi^m)}(G) = 1$. In this case, we have $\dim_{\ka}(\pi^{m-1}(G)) = \dim_{\ka}(\pi^{m-1}T/\pi^{a}T) = 0$, finishing the proof.
\end{proof}

\hspace{3mm} When we use Theorem \ref{thm:fw_fiberwise}, we need to decipher $\beta_1^R(G)$. To further control this number, we need the following property of $R = \bZ_p[t]/(P(t))$, first observed by the first author and Yu \cite[Lemma 2.2]{CY}. We give a different proof; it is considerably shorter because it utilizes the theory of minimal resolutions.

\begin{lem} Suppose that the reduction $\bar{P}(t)$ of $P(t)$ modulo $p$ is given by $\bar{P}(t) = \bar{Q}(t)^m$ for some monic irreducible $\bar{Q}(t) \in \bF_p[t]$ and $m \in \bZ_{\geq 1}$. Then any finite-length $R$-module $G$ satisfies
    \begin{equation}
        \beta_1^R(G)\geq \beta_0^R(G).
    \end{equation}
\label{lem:ext_geq_hom}
\end{lem}
\begin{rmk}
The above lemma no longer holds if $\Zp$ is replaced by $\bZ/p^k\bZ$ with any $k \in \bZ_{\geq 1}$, even when $P(t)=t$, as can be seen from Lemma \ref{lem:betti_dvr_quotient}.
\end{rmk}

\begin{proof}[Proof of Lemma \ref{lem:ext_geq_hom}] We note that the hypotheses imply that $R = \Zp[t]/(P(t))$ is local. Let $b_i=\beta_i^R(G)$ and fix a monic lift $Q(t) \in \bZ_p[t]$ of $\bar{Q}(t)$. By choosing a minimal resolution of $G$, there exists a matrix $A\in \M_{b_0\times b_1}(R)$ such that $\cok(A)\simeq_R G$. In particular, $\cok(A)$ is of finite length. We shall find an $R$-algebra $K$ that is a field such that $\cok_K(A)=0$. If so, the existence of a $b_0\times b_1$ matrix $A$ over $K$ that gives rise to a surjective $K$-linear map would imply $b_1\geq b_0$.

\hspace{3mm} Recall that $\Zp[t]$ is a unique factorization domain. In particular, the polynomial $P(t)$ admits a factorization into monic irreducible polynomials in $\Zp[t]$. Let $F(t)$ be a monic irreducible factor of $P(t)$ in $\bZ_p[t]$, and consider the ring $S:=\Zp[t]/(F(t))$, which is a quotient of $R$. More importantly, the ring $S$ is a local domain that is not a field. (If $S$ were a field, then $F(t)R$ would be a maximal ideal of $R$. On the other hand, the unique maximal ideal of $R$ is $\m=(p, Q(t))R$, which is not $F(t)R$ because $p\notin F(t)R$.) Let $K$ be the fraction field of $S$ and view $K$ as an $R$-algebra. We now claim that $\cok_K(A)=0$. 

\hspace{3mm} It suffices to show that $G \otimes_R K=0$. Let $G' := G \otimes_R S$. Note that $\mf{m}S$ is the maximal ideal of $S$ because $S$ is a quotient of $R$. Note that, as an $R$-module, $G'$ is of finite length because it is a quotient of $G$. Thus, there exists $N \geq 0$ such that $\m^NG = 0$ so that $(\m^N S) G' = 0$ as an $S$-module. Since $S$ is a domain that is not a field, there exists $x\in \m^N S\sm \set{0}$. We have $xG'=0$, so that $x$ annihilates $G' \otimes_S K$ as well. But $x$ is invertible in $K$, which implies $G \otimes_R K \simeq_K G'\otimes_S K = 0$, and the proof is complete.
\end{proof}

\subsection{Proofs of Proposition \ref{prep'} and Theorem \ref{Ha'} assuming Theorems \ref{thm:fw_general} and \ref{thm:fw_fiberwise}}
We are now ready to prove Proposition \ref{prep'} and Theorem \ref{Ha'} assuming Theorems \ref{thm:fw_general} and \ref{thm:fw_fiberwise}.

\begin{proof}[Proofs of Proposition \ref{prep'} and Theorem \ref{Ha'} assuming Theorems \ref{thm:fw_general} and \ref{thm:fw_fiberwise}.]
Recall the factorization of $\bar P(t)$ in \eqref{fac}. By Hensel's lemma, there exists monic $Q_1(t), \dots, Q_l(t)\in \Zp[t]$ such that $P(t) = Q_1(t) \cdots Q_l(t)$ and $Q_j(t) \equiv \bar P_j(t)^{m_j}\ppmod p$. Let $R_j:=\Zp[t]/(Q_j(t))$. By the Chinese remainder theorem, we have $R \simeq_R R_1\times\dots\times R_l$ given by $x \mapsto (x \mod (Q_1), \dots, x \mod (Q_l))$. Applying this particular isomorphism, we have
\[\M_n(R) \simeq_R \M_n(R_1) \times \cdots \times \M_n(R_l),\]
and the Haar measure on $\M_n(R)$ is the product measure of the Haar measures of $\M_n(R_j)$ because of the uniqueness of the Haar measure. Hence, to prove Proposition \ref{prep'} and Theorem \ref{Ha'}, it suffices to prove them for the case $l=1$. (More details of this reduction can be found in \cite[\S 2.1]{CLS} by replacing $P_j(t)$ in the citation with $Q_j(t)$.) Therefore, we may assume from now on that $\bar{P}(t) = \bar{Q}(t)^{m}$ for some monic irreducible $\bar{Q}(t) \in \bF_p[t]$ and $m \in \bZ_{\geq 1}$. In particular, the ring $R = \Zp[t]/(P(t))$ is local. Write $d := \deg(\bar{Q})$ and $q := p^d$.

\hspace{3mm} We first assume (1) and then show (2) in Proposition \ref{prep'}. Lemma \ref{lem:ext_geq_hom} implies that $\beta_0^R(G) \leq \beta_1^R(G)$. Theorem \ref{thm:fw_general} with $u=0$ implies $\beta_0^R(G) \geq \beta_1^R(G)$. Thus, we have
\[|\Hom_{\bZ_{p}[t]}(G, \bF_{q})| = \beta_0^R(G)=\beta_1^R(G) = |\Ext_{R}^{1}(G, \bF_{q})|,\]
which is (2).

\

\hspace{3mm} Next, we assume that (2) from Proposition \ref{prep'} implies the conclusion of Theorem \ref{Ha'}. Taking $u=0$ and $\fa = pR$ in Theorem \ref{thm:fw_fiberwise} (with $J_n = \bar X$), we have $b_0 = b'_0$ and thus, applying Lemma \ref{lem:betti_dvr_quotient} (and the discussion before that), we have
\begin{align*}
b_0 - b'_1 &= b'_0 - b'_1 \\
&= \beta_0^{R/pR}(G/pG) - \beta_1^{R/pR}(G/pG) \\
&= \dim_{\bF_q}(\bar{Q}(t)^{m-1} G/pG) \\
&= u_1(G/pG).
\end{align*}
Since (2) from Proposition \ref{prep'} implies $b_0 = b_1$, we obtain Theorem \ref{Ha'}.

\

\hspace{3mm} Finally, we assume (2) and then show (1) in Proposition \ref{prep'}. We already know that (2) implies the conclusion of Theorem \ref{Ha'}. Then 
\[\underset{Z \in \M_n(R)}{\Prob}(\cok(Z) \simeq_{\bZ_p[t]} G \text{ and } Z \equiv J_n \ppmod{p}) \neq 0,\]
so we get the existence of such $Z$. This finishes the proofs of Proposition \ref{prep'} and Theorem \ref{Ha'} assuming Theorem \ref{thm:fw_fiberwise}.
\end{proof}

\hspace{3mm} For the rest of the section, we prove Theorem \ref{thm:fw_general} and Lemma \ref{lem:equidistribution_residue} (which imply Theorem \ref{thm:fw_fiberwise}) that we have deferred. Then by the previous subsection, we would establish Proposition \ref{prep'} and Theorem \ref{Ha'}, which would only leave Theorem \ref{equi} to finish the proof of Theorem \ref{main}. We collect some preliminaries in commutative algebra needed in the proofs.

\subsection{Preliminaries in commutative algebra for proofs of Theorem \ref{thm:fw_general} and Lemma \ref{lem:equidistribution_residue}} The proofs of Theorem \ref{thm:fw_general} and Lemma \ref{lem:equidistribution_residue} in the DVR case relies on the classification of finitely generated modules, namely, the Smith normal form. In order to generalize the proof to a Noetherian local ring that is not a DVR, we need to show that some nice consequences of the Smith normal form persist even in its absence. The following lemma is the key ingredient in the proof of Lemma \ref{lem:row_col_equiv} and Lemma \ref{lem:betti_bound}. The former is used in the proof of Lemma \ref{lem:equidistribution_residue}, and the latter is part of Theorem \ref{thm:fw_general} and is crucially used in Lemma \ref{lem:fw_general_step3}, the last step of the proof of Theorem \ref{thm:fw_general}. Denote by $\Sur_R(G, H)$ the set of $R$-linear surjections from $G$ to $H$, given $R$-modules $G$ and $H$.

\begin{lem}\label{lem:sur_homogeneity}  Let $(R, \m, \ka)$ be any Noetherian local ring, and $G$ be a finitely generated $R$-module. Suppose that $n \geq \beta_0^R(G)$. Then $\GL_n(R)$ acts on $\Sur_R(R^n, G)$ transitively: for any $F_1, F_2 \in \Sur_R(R^n, G)$, there is $g \in \GL_n(R)$ such that $F_2=F_1\circ g$.
\end{lem}

\begin{proof} Let $r = \dim_{\ka}(G /\m G) = \beta^R_0(G)$, the minimal number of generators for $G$. Fix an $R$-linear surjection $\varphi:R^r\onto G$. Recall that free modules are projective. That is, any diagram of $R$-modules below lifts:
    \begin{equation*}
    \begin{tikzcd}
    {R^n} \\
    A & B
    \arrow[from=1-1, to=2-2]
    \arrow[two heads, from=2-1, to=2-2]
    \arrow[dashed, from=1-1, to=2-1]
    \end{tikzcd}
    \end{equation*}
    
    Therefore, we have $R$-linear maps $F_1', F_2' : R^n\to R^r$ such that the diagram
    \begin{equation*}
    \begin{tikzcd}
	{R^n} \\
	& {R^r} & G, \\
	{R^n}
	\arrow["{F_1}", from=1-1, to=2-3]
	\arrow["{F_2}"', from=3-1, to=2-3]
	\arrow["\varphi", from=2-2, to=2-3]
	\arrow["{F_1'}"', from=1-1, to=2-2]
	\arrow["{F_2'}", from=3-1, to=2-2]
\end{tikzcd}
    \end{equation*}
is commutative.

\hspace{3mm} Tensoring the diagram with $\ka = R/\m$, the map $\varphi$ becomes an isomorphism of $\ka$-vector spaces by the assumption that the minimal number of generators of $G$ is $r$. For $i=1,2$, since the mod-$\m$ reduction of $F_i$ is surjective, so is the mod-$\m$ reduction $\bbar F_i'$ of $F_i'$. By Nakayama's lemma, $F_i'$ is surjective. Hence, we may replace $G$ by $R^r$, and we have reduced to the case where $G$ is a free module $R^r$, and $F_1,F_2$ are surjective $r\times n$ matrices.

\hspace{3mm} We now claim that there exists $g\in \GL_n(R)$ such that $F_2=F_1 g$. For $i=1,2$, by right-multiplying $F_i$ by a matrix in $\GL_r(\ka)$ if necessary, we may assume the first $r$ columns of $\bbar F_i$ span $\ka^r$. Write
    \begin{equation*}
        F_1=\begin{bmatrix}U & A\end{bmatrix} \text{ and } \;\;F_2=\begin{bmatrix}V & B\end{bmatrix}, \text{ where } \;\; U,V\in \Mat_r(R) \text{ and } \;\; A,B\in \Mat_{n-r}(R).
    \end{equation*}

    By our assumption, $U,V$ are invertible mod $\m$, thus invertible over $R$. Considering
    \begin{equation*}
        g := \begin{bmatrix}
            U^{-1}V &U^{-1}(B-A) \\
            0 & I_{n-r}
        \end{bmatrix}\in \GL_n(R),
    \end{equation*}
    we have $F_2=F_1g$ as desired.    
\end{proof}

\begin{rmk} Lemma \ref{lem:sur_homogeneity} can also be deduced from \cite[Theorem 20.2]{E}.
\end{rmk}

\hspace{3mm} Theorem \ref{thm:fw_general} concerns all matrices with a fixed cokernel up to isomorphism. We now show that all such matrices are row-column-equivalent, as they are in the DVR case. More precisely, we have the following. (Technically, we do not need it for the proof of Theorem \ref{thm:fw_general}, but we use it in the proof of Lemma \ref{lem:equidistribution_residue}.)

\begin{lem}\label{lem:row_col_equiv}
    Let $R$ be a Noetherian local ring and $m, n \in \bZ_{\geq 1}$. Consider any two $m\times n$ matrices over $R$ or equivalently, $R$-linear maps $A,B:R^n\to R^m$. Then
    \begin{enumerate}
        \item $\im(A)=\im(B)$ as submodules of $R^m$ if and only if $A$ and $B$ are \textbf{column-equivalent}, namely, $Ag=B$ for some $g\in \GL_n(R)$.
        \item Let $N_1,N_2$ be submodules of $R^m$. Then $R^m/N_1$ and $R^m/N_2$ are isomorphic as $R$-modules if and only if $N_1$ and $N_2$ are \textbf{row-equivalent}, namely, $gN_1=N_2$ for some $g\in \GL_m(R)$.
        \item $\cok(A)$ and $\cok(B)$ are isomorphic as $R$-modules if and only if $A$ and $B$ are \textbf{row-column-equivalent}, namely, $gAg'=B$ for some $g\in \GL_m(R)$ and $g'\in \GL_n(R)$.
    \end{enumerate}
\end{lem}
\begin{proof}~
    \begin{enumerate}
        \item The backward implication is trivial. For the forward implication, write $M=\im (A)=\im (B) \subeq R^m$ so that we can consider $A,B\in \Sur_R(R^n,M)$. By Lemma \ref{lem:sur_homogeneity}, there is $g\in \GL_n(R)$ such that $A\circ g=B$ as maps from $R^n$ to $M$. Composed with the inclusion map of $M$ into $R^m$, we have $Ag=B$ as matrices.
        \item The backward implication is evident, since $g\in \GL_m(R)$ induces an isomorphism from $R^m/N_1$ to $R^m/gN_1$. For the forward implication, let $M = R^m/N_1 \simeq_R R^m/N_2$. Then we have the following commutative diagram of $R$-linear maps, whose rows are exact:
        \begin{equation*}
        \begin{tikzcd}
	0 & {N_1} & {R^m} & M & 0 \\
	0 & {N_2} & {R^m} & M & 0,
	\arrow[from=1-1, to=1-2]
	\arrow[from=1-2, to=1-3]
	\arrow[from=1-3, to=1-4]
	\arrow[from=1-4, to=1-5]
	\arrow[from=2-1, to=2-2]
	\arrow[from=2-2, to=2-3]
	\arrow[from=2-3, to=2-4]
	\arrow[from=2-4, to=2-5]
	\arrow["\mathrm{id}"', from=1-4, to=2-4]
	\arrow["g", from=1-3, to=2-3]
        \end{tikzcd}
        \end{equation*}
        where $g$ is constructed from Lemma \ref{lem:sur_homogeneity} applied to the two quotient maps $R^m\to M$ induced by $N_1$ and $N_2$. Therefore, we have $gN_1=N_2$.
        \item The backward implication is trivial. For the forward implication, if $\cok(A) \simeq_R \cok(B)$, then $N_1=\im(A)$ and $N_2=\im(B)$ satisfy the assumption of (2), so $\im(B)=g\cdot \im(A)$ for some $g\in \GL_m(R)$. Thus $\im(B)=\im(gA)$, so by (1), there is $g'\in \GL_n(R)$ such that $B=gAg'$.
    \end{enumerate}
This finishes the proof.
\end{proof}

\hspace{3mm} The following lemma is a part of Theorem \ref{thm:fw_general}.

\begin{lem}\label{lem:betti_bound}
    Let $R$ be a Noetherian local ring and $G$ a finitely generated $R$-module. Write $b_i:=\beta_i^R(G)$. For integers $n \geq 1$ and $u\geq 0$, if there exists $X\in \M_{n\times (n+u)}(R)$ with $\cok(X)\simeq_R G$, then $n\geq b_0\geq b_1-u$.
\end{lem}
\begin{proof}
    Consider the exact sequence
    \begin{equation*}
        R^{n+u} \map[X] R^{n} \map[A] G \to 0,
    \end{equation*}
where $A$ is the $R$-linear map given by $R^n \tra R^n/XR^{n+u} = \cok(X) \simeq_R G$, and let $M := \ker(A) \sub R^n$. From the existence of the surjection $A$, it follows that $n\geq b_0$. From the existence of the $R$-linear surjection $X:R^{n+u}\to M$, it follows that $n+u\geq \beta_0^R(M)$. To prove $b_0\geq b_1-u$, it suffices to show that
\[\beta_0^R(M)=n+b_1-b_0.\]

\hspace{3mm} By Lemma \ref{lem:sur_homogeneity}, if $A'$ is any $R$-linear surjection from $R^n$ to $M$, then $\ker(A')$ is isomorphic to $M = \ker(A)$ and thus $\beta_0^R(\ker(A')) = \beta_0^R(M)$. We construct a convenient choice of $A'$ below. Pick a minimal resolution
    \begin{equation*}
        \dots\to R^{b_1} \to R^{b_0} \map[A_0] G \to 0
    \end{equation*}
    of $G$, and write $M_0:=\ker(A_0)$. Then $\beta_0^R(M_0)=b_1$ by the definition of a minimal resolution. Now construct $A' := A_0\oplus 0: R^{b_0}\oplus R^{n-b_0} \onto G$, then $\ker(A') = M_0 \oplus R^{n-b_0}$. It follows that
\[\beta_0^R(M)=\beta_0^R(M_0\oplus R^{n-b_0})=b_1 + (n - b_0) = n + b_1 - b_0,\]
as desired.
\end{proof}

\hspace{3mm} Similar to the proof by Friedman--Washington \cite{FW} in the DVR case, we reduce the Haar-measure statement in Theorem \ref{thm:fw_general} into a counting statement by passing to a sufficiently large finite quotient of $R$. We need the following lemmas in the reduction step. %In the statement, note that for an ideal $\fa\subeq R$, an $R/\fa$-module is just an $R$-module annihilated by $\fa$.

\begin{rmk} In the reduction step in the proof of Theorem \ref{thm:fw_general}, we shall apply Lemma \ref{lem:betti_wrt_quotient_ring} (3) with $\fa=\m^L$ and $\fb=\m^{L-1}$, where $L$ is large enough so that $\fb M=0$.
\end{rmk}

\subsection{Proof of Theorem \ref{thm:fw_general}}
\label{sec:fw_general_proof}

The ``only if'' direction of the existence statement of Theorem \ref{thm:fw_general} follows from Lemma \ref{lem:betti_bound}. Once the probability formula \eqref{eq:fw_general} of Theorem \ref{thm:fw_general} is proved, the ``if'' direction of the existence statement follows from the fact that the probability is nonzero. Hence it suffices to prove \eqref{eq:fw_general}, under the assumption that $n \geq b_0 \geq b_1 - u$, where $b_i := \beta_i^R(G)$. We carry this out in three steps.

\begin{lem}[Step 1]\label{lem:fw_general_step1}
    To prove \eqref{eq:fw_general}, it suffices to prove the case when $R$ is of finite size.
\end{lem}
\begin{proof}
    Assume \eqref{eq:fw_general} with the hypothesis $n \geq b_0 \geq b_1 - u$ holds for any finite-sized local ring $R$. Now, let $R$ and $G$ be given as in Theorem \ref{thm:fw_general}, where $R$ is not necessarily of finite size. Suppose that $n \geq \beta_0(G) \geq \beta_1(G) - u$. Since $G$ is of finite length, there exists $L \in \bZ_{\geq 2}$ such that $\m^{L-1}G=0$. For any $X\in \Mat_{n\times (n+u)}(R)$, we denote by $\bar X$ the residue class of $X$ modulo $\m^L$. Since 
\[\cok_{R/\m^L}(\bar X) \simeq_R \cok(X)\otimes_R R/\m^L \simeq_R \cok(X)/\m^L\cok(X),\]
by Lemma \ref{lem:quotient_determines_original} with $\fb=\m^{L-1}$ and $M = \cok(X)$, we have $\cok(X)\simeq_R G$ if and only if $\cok_{R/\m^L}(\bbar X) \simeq_R G$. Moreover, by Lemma \ref{lem:betti_wrt_quotient_ring} (3) with $\fb = \m^{L-1}$, we have $\beta_i^{R/\m^{L}}(M)=\beta_i^R(M)$ for $i=0,1$. Hence, both sides of \eqref{eq:fw_general} are unchanged if we replace $R$ by $R/\m^{L}$ everywhere. Therefore, the equality in \eqref{eq:fw_general} holds by our assumption.
\end{proof}

\hspace{3mm} For the rest of the proof, we assume $R$ is a finite-sized local ring. Our goal is to count the cardinality of $\set{X\in \M_{n\times (n+u)}(R): \cok(X) \simeq_R G}$. We divide this in two steps: we first count the number of all possible images of $X$ in the set we count, and then count the number of such $X$ with a given image. We may immediately notice that the image of any such $X$ must be a submodule $M \sub R^n$ such that $R^n/M \simeq_R G$, and any such matrix $X$ with a given image $M$ corresponds to an $R$-linear surjection from $R^{n+u}$ to $G$. The following lemma is due to Cohen and Lenstra \cite[Proposition 3.1 (iii)]{CL}:

\begin{lem}[Step 2]\label{lem:fw_general_step2} Let $(R, \m, \Fq)$ be a local ring of finite size and $G$ a finite-sized $R$-module. If $n \geq \beta_0^R(G) = b_0$, then the number of submodules of $R^n$ with quotient $G$ is given by
    \begin{equation*}
        \#\set{M \leqslant R^n: R^n/M \simeq_R G} = \frac{\abs{G}^n}{\abs{\Aut_R(G)}} \prod_{i=n-b_0+1}^n (1-q^{-i}).
    \end{equation*}
\end{lem}
\begin{proof}
    We note that $\set{M \ls R^n: R^n/M \simeq_R  G}$ can be identified with the set of $\Aut_R(G)$-orbits of $\Sur_R(R^n, G)$, where $\Aut_R(G)$ acts on $\Sur_R(R^n, G)$ by composition: that is, given any $\phi_1, \phi_2 \in \Sur_R(R^n, G)$, we have $\ker(\phi_1) = \ker(\phi_2)$ if and only if $\phi_2 = \sg \circ \phi_1$ for some $\sg \in \Aut_R(G)$. The action is free: if $A \in \Sur_R(R^n, G)$ and $\sigma\in \Aut_R(G)$ satisfies $\sigma \circ A = A$, then $\sigma$ must be the identity because $A$ is surjective. Therefore, the orbit-stabilizer theorem implies that every orbit has the size $\abs{\Aut_R(G)}$, so
    \begin{equation*}
        \#\set{M \leqslant R^n: R^n/M \simeq_R G} = \frac{\abs{\Sur_R(R^n, G)}}{\abs{\Aut_R(G)}}.
    \end{equation*}

    We now compute $\abs{\Sur_R(R^n, G)}$. By Nakayama's lemma, an $R$-linear map $A : R^n\to G$ is surjective if and only if its mod-$\m$ reduction $\bar A: \Fq^n \to G/\m G$ is surjective. Therefore, the probability that a uniformly random $A\in \Hom_R(R^n, G)$ be surjective is
    \begin{equation}\label{eq:sur_prob}
        \frac{\abs{\Sur_ \Fq(\Fq^n, \Fq^{b_0})}}{\abs{\Hom_\Fq(\Fq^n,  \Fq^{b_0})}} = \prod_{i=n-b_0+1}^n (1-q^{-i}).
    \end{equation}

    Since $\abs{\Hom_R(R^n, G)}=\abs{G}^n$, the result follows.
\end{proof}

\begin{lem}[Step 3]\label{lem:fw_general_step3}
    Assume $(R, \m, \Fq)$ is a local ring of finite size and $M \sub R^n$ is a submodule. Let $G :=R ^n/M$. Then
    \begin{equation*}
        \abs{\Sur_R(R^{n+u},M)}=\frac{\abs{R}^{n(n+u)}}{\abs{G}^{n+u}}\prod_{i=u+b_0-b_1}^{n+u} (1-q^{-i}),
    \end{equation*}
    where $b_i=\beta_i^R(G)$. In particular, the quantity depends only on the isomorphism class of $G$, but not on $M$.
\end{lem}
\begin{proof}
    From the proof of Lemma \ref{lem:betti_bound}, we have $\beta_0^R(M)=n+b_1-b_0$ by taking $A$ to be the quotient map $R^n \tra R^n/M \simeq_R G$ so that $M = \ker(A)$. By the same argument involving \eqref{eq:sur_prob}, we have
    \begin{equation*}
        \abs{\Sur(R^{n+u}, M)}=\abs{M}^{n+u}\prod_{i=n+u-\beta_0^R(M)}^{n+u} (1-q^{-i}).
    \end{equation*}

    The desired formula then follows because $|R|^{n(n+u)} = |R^n|^{n+u} = |G|^{n+u}|M|^{n+u}$.
\end{proof}

\hspace{3mm} We are now ready to show Theorem \ref{thm:fw_general}:

\begin{proof} [{Proof of Theorem \ref{thm:fw_general}}] By Lemma \ref{lem:fw_general_step1}, we may assume that $R$ is of finite size. It remains to prove \eqref{eq:fw_general} under the assumption $n\geq b_0\geq b_1-u$. By Lemma \ref{lem:fw_general_step2} and Lemma \ref{lem:fw_general_step3}, we have
    \begin{align*}
        \underset{X\in \M_{n\times (n+u)}(R)}\Prob(\cok(X) \simeq_R G) & = \frac{1}{\abs{R}^{n(n+u)}} \# \set{X\in \M_{n\times (n+u)}(R): \cok(X) \simeq_R G} \\
        &= \frac{1}{\abs{R}^{n(n+u)}} \parens*{\frac{\abs{G}^n}{\abs{\Aut_R (G)}} \prod_{j=n-b_0+1}^n (1-q^{-j}) } \parens*{\frac{\abs{R}^{n(n+u)}}{\abs{G}^{n+u}}\prod_{i=u+b_0-b_1}^{n+u} (1-q^{-i})} \\
        &= \frac{1}{\abs{\Aut_R (G)} \abs{G}^u} \prod_{i=u+b_0-b_1}^{n+u} (1-q^{-i}) \prod_{j=n-b_0+1}^n (1-q^{-j}),
    \end{align*}
    which is \eqref{eq:fw_general}.
\end{proof}

\subsection{Proof of Lemma \ref{lem:equidistribution_residue}}\label{sec:equidistribution_residue_proof} We now prove Lemma \ref{lem:equidistribution_residue}:

\begin{proof}[Proof of Lemma \ref{lem:equidistribution_residue}] Denote by $P(G|\bar X)$ the conditional probability in \eqref{eq:fw_fiberwise}. Suppose that $\bbar {X_1},\bbar {X_2}\in \M_n(R/\fa)$ satisfy $\cok(\bbar {X_1}) \simeq_R G/\fa G \simeq_R \cok(\bbar {X_2})$. We shall prove that $P(G|\bbar {X_1})=P(G|\bbar {X_2})$.
   
\hspace{3mm} By Lemma \ref{lem:row_col_equiv} (3) applied to the ring $R/\fa$, there exist $\bar g\in \GL_n(R/\fa)$ and $\bar g'\in \GL_m(R/\fa)$ such that $\bar g \bbar {X_1} \bar g'=\bbar {X_2}$. Pick any lifts $g\in \M_n(R)$ and $g'\in \M_m(R)$ of $\bar g$ and $\bar g'$, respectively. Since invertibility can be tested modulo $\m$, the matrices $g,g'$ must be invertible.

\hspace{3mm} Consider the map
    \begin{equation*}
        \begin{aligned}       
            \set{X_1\in \M_n(R):X_1\equiv \bbar {X_1}\pmod \fa} & \to \set{X_2\in \M_n(R):X_2\equiv \bbar{X_2}\pmod \fa} \text{ given by} \\
            X_1 &\mapsto gX_1g',
        \end{aligned}
    \end{equation*}
    which is well-defined since $\bar g \bbar {X_1} \bar g'=\bbar {X_2}$. This map is a measure-preserving bijection because it is a restriction of an $R$-linear automorphism of $\M_n(R)$ and the Haar measure on $\M_n(R)$ is unique. By its definition, this map preserves the cokernel up to $R$-linear isomorphism, so $P(M|\bbar {X_1})=P(M|\bbar {X_2})$.
\end{proof}

\hspace{3mm} Hence, to show Theorem \ref{main}, it remains to show Theorem \ref{equi}. In the next section, we shall reduce Theorem \ref{equi} into another lemma, which is proven in \S \ref{HaEnd}.

\section{Reduction of proof of Theorem \ref{equi}}

\hspace{3mm} The high-level idea of this section is originated from \cite{CK} and \cite{CLS}. Write $\bar{X} := A_n$ in Theorem \ref{equi} since our $n$ is fixed throughout this section. Write $R := \bZ_p[t]/(P(t))$ and $d := \deg(P)$. To prove Theorem \ref{equi}, it suffices to construct a measure-preserving bijection
\begin{align*}
\set{X\in \M_n(\Zp)_{\bar X}:& \cok_R(X+\bar t(pY_1-I_n)+\bar t^2pY_2+\dots+\bar t^{d-1}pY_{d-1}) \simeq_R G} \\
&\to \set{X'\in \M_n(\Zp)_{\bar X}: \cok_R(X'-\ol t I_n) \simeq_R G},
\end{align*}
given the hypotheses of Theorem \ref{equi}.

\hspace{3mm} To achieve this, we note that $\cok(ZU) \simeq_R \cok(Z)$ for any $Z\in \M_n(R)$ and $U\in \GL_n(R)$, so it suffices to construct a measure-preserving bijection $\Phi: \M_n(\Zp)_{\bar X} \to \M_n(\Zp)_{\bar X}$ such that whenever $X'=\Phi(X)$, there exists $U\in \GL_n(R)$ such that
\begin{equation}\label{eq:equidist_map_original}
    \parens*{X+\bbar t(pY_1-I_n)+\bbar t^2pY_2+\dots+\bbar t^{d-1}pY_{d-1}}U=X'-\ol t I_n.
\end{equation}

\hspace{3mm} When $d=2$, the first author and Kaplan \cite[p.645]{CK} observed that we can take $\Phi(X) = X(I_n - pY_1)^{-1}$ with $U=(I_n-pY_1)^{-1}$. Note that the inverse of $\Phi$ is given by $\Phi^{-1}(X') = X'(I_n - pY_1)$. 

\hspace{3mm} When $d\geq 3$, as observed by the first author, Liang, and Strand in \cite[Remark 3.8]{CLS}, a simple choice of $\Phi$ is no longer available. Nevertheless, we show that such $\Phi$ exists through an algorithmic approach. For clarity, we state our claim as a lemma, which slightly cleans up the hypotheses in Theorem \ref{equi} and \eqref{eq:equidist_map_original}.

\begin{lem}\label{lem:equidist_map}
    Let $P(t)\in \Zp[t]$ be monic of degree $d \geq 2$ and $pY_2, \dots, pY_{d-1}\in p\M_n(
    \Zp)$. Let $R=\Zp[t]/P(t)$. Then there exists a Haar measure-preserving bijection $\Phi: \M_n(\Zp) \to \M_n(\Zp)$ such that whenever $X'=\Phi(X)$, we have $X\equiv X'\ppmod p$ and
    \begin{equation}\label{eq:equidist_map_clean}
        \parens*{X+\bbar tI_n+\bbar t^2pY_2+\dots+\bbar t^{d-1}pY_{d-1}}U = X'+\ol t I_n
    \end{equation}
    for some $U\in \GL_n(R)$ potentially depending on $X$.
\end{lem}

\begin{proof}[Proof that Lemma \ref{lem:equidist_map} implies Theorem \ref{equi}] We assume Lemma \ref{lem:equidist_map} and then establish \eqref{eq:equidist_map_original}. Given the hypotheses of Theorem \ref{equi}, we note
\begin{align*}
&X + \bar{t}(pY_1 - I_n) + \bar{t}^2pY_2 + \cdots + \bar{t}^{d-1}pY_{d-1} \\
&= (X(pY_1 - I_n)^{-1} + \bar{t}I_n + \bar{t}^2pY_2(pY_1 - I_n)^{-1} + \cdots + \bar{t}^{d-1}pY_{d-1}(pY_1 - I_n)^{-1})(pY_1 - I_n).
\end{align*}
Applying Lemma \ref{lem:equidist_map} by replacing $X$ with $X(pY_1 - I_n)^{-1}$ and $X'$ with $-X'$, which makes sense because $X(pY_1 - I_n) ^{-1}\equiv -X' \ppmod{p}$, we may find some $V \in \GL_n(R)$ such that
\[(X(pY_1 - I_n)^{-1} + \bar{t}I_n + \bar{t}^2pY_2(pY_1 - I_n)^{-1} + \cdots + \bar{t}^{d-1}pY_{d-1}(pY_1 - I_n)^{-1})V = - X' + \bar{t} I_n.\]
Then taking $U = -(pY_1 - I_n)^{-1}V$, we obtain \eqref{eq:equidist_map_original}.
\end{proof}

\hspace{3mm} Thus, to prove Theorem \ref{main}, it remains to prove Lemma \ref{lem:equidist_map}. Before we start the proof of Lemma \ref{lem:equidist_map}, we give the simplest nontrivial example to illustrate the idea and its apparent difficulties.

\begin{exmp}\label{exmp:division} Let $d=3$ and suppose we are given $f=X+\ol t I_n+\ol t^2 pY_2$, where $X\in \M_n(\Zp)$ and $pY_2\in p\M_n(\Zp)$. We say $g \in \M_n(R)$ is \textbf{equivalent} to $f$ if $g=fU$ for some $U\in \GL_n(R)$. We wish to find an element without $\ol t^2$ or higher terms that is equivalent to $f$. An obvious attempt is to keep updating $f$ by an equivalent element, each step getting rid of some higher terms of $f$, and see if this process eventually terminates. For example, an initial candidate could be
\[f(I_n-\ol tpY_2)=X+\ol t (I_n-XpY_2)-\ol t^3 p^2Y_2^2.\]
Correcting the linear coefficient, we get 
\[f(I_n - \ol t pY_2)(I_n-XpY_2)^{-1}=X(I_n-XpY_2)^{-1} + \ol t I_n - \ol t^3 p^2Y_2^2 (I_n-XpY_2)^{-1}.\] We are making progress since the coefficient of $\bar{t}^3$ is a multiple of $p^2$, so the higher terms are more divisible by $p$ than before. However, if we repeat this process again, we get
    \begin{equation*}
    \begin{multlined}
        f(I_n - \ol t pY_2)(I_n-XpY_2)^{-1}(I_n + \ol t^2 p^2Y_2^2 (I_n - XpY_2)^{-1})\\ = X(I_n-XpY_2)^{-1} + \ol t I_n + \ol t^2 X(I_n-XpY_2)^{-1} p^2Y_2^2 (I_n-XpY_2)^{-1} - \ol t^5 p^2Y_2^2 (I_n-XpY_2)^{-1} p^2Y_2^2 (I_n-XpY_2)^{-1}.
    \end{multlined}
    \end{equation*}
Here, the higher terms (i.e., $\bar{t}^2$ or higher) are still only known to be divisible by $p^2$. The reader is encouraged to repeat the process again, and find that the higher terms are divisible by $p^3$ after the process. 
\end{exmp}

\hspace{3mm} In fact, the process in Example \ref{exmp:division} turns out to ``converge,'' although it is unclear how to prove it. When $d > 3$, the situation is even more convoluted. Our goal is to is systematically describe an algorithm to establish such a convergence. Furthermore, the construction of $\Phi(X)$ is extremely complicated, which makes it almost impossible to directly show that $\Phi$ is a bijection. In the next section, we deal with this complication by mimicking a common technique in commutative algebra, called the \emph{Weierstrass preparation theorem}, for our noncommutative ring $\M_n(\bZ_p)$.

\section{A noncommutative Weierstrass preparation theorem and proof of Lemma \ref{lem:equidist_map}} \label{HaEnd}

\subsection{A noncommutative Weierstrass preparation theorem} In commutative algebra, the \textbf{Weierstrass preparation theorem} states that given a complete local ring $(A, \m)$, if $f(t) = a_0 + a_1 t + a_2 t^2 + \cdots \in A\llb t \rrb$ with not all $a_i$ are in $\m$, then there is a unique unit $u(t) \in A\llb t \rrb$ and a polynomial $F(t) = t^s + b_{s-1}t^{s-1} + \cdots + b_1 t + b_0 \in A[t]$ with $b_i \in \m$ such that $f(t) = u(t)F(t)$.

\hspace{3mm} For our purpose, our ring is $A := \M_n(\bZ_p)$ which is a non-commutative ring for any $n \geq 2$. We are fixing our $n \in \bZ_{\geq 1}$ in this section.
\begin{setting}\label{setting:weierstrass} We note that $A = \M_n(\bZ_p)$ satisfies the following properties:
    \begin{enumerate}
    \item $\bigcap_{N=1}^\infty p^N A = 0$.
    \item If $(a_n)_{n \in \bZ_{\geq 0}}$ is a sequence in $A$ such that for any $N \in \bZ_{\geq 0}$, the sequence $(a_n \bmod {p^N})$ eventually stabilizes, then the sequence $(a_n)_{n \in \bZ_{\geq 0}}$ converges in $A$. (That is, there exists $a \in A$ such that for any $N \in \bZ_{\geq 0}$, there exists $m \in \bZ_{\geq 0}$ such that $a_n \equiv a  \ppmod {p^N}$ whenever $n \geq m$.)
    \end{enumerate}
\end{setting}    

\hspace{3mm} Our theorem will take place in the ring $\Athat$ defined below.

\begin{defn}
    Let $A[t]$ and $A\llb t \rrb$ be the polynomial ring and the power series ring over $A$ generated by a variable $t$ that commutes with $A$. Define $\Athat$ to be the subring of $A\llb t \rrb$ given by
    \begin{equation}
        \Athat := \set*{\sum_{l=0}^{\infty}C_l t^l : C_l \in A\text{ and $\lim_{l \to \infty} C_l=0$}}.
    \end{equation}
For $A[t]$ and $A\llb t \rrb$ we use the product topology induced from $A$. Then $\Athat \sub A\llb t \rrb$ gets the subspace topology.
\end{defn}

\begin{lem}\label{com} With respect to the $p$-adic topology, the ring $\Athat$ is complete. 
\end{lem}

\begin{proof} Let $(F_j(t))_{j \in \bZ_{\geq 0}}$ be a Cauchy sequence in $\Athat$. Write
\[F_j(t) = C_{j0} + C_{j1}t + C_{j2}t^2 + \cdots.\]
Since $(F_j(t))_{j \in \bZ_{\geq 0}}$ is Cauchy in $\Athat$, for every $l \in \bZ_{\geq 0}$, the sequence $(C_{jl})_{j \in \bZ_{\geq 0}}$ is Cauchy in $A = \M_n(\bZ_p)$, which is complete with respect to its $p$-adic topology. Thus, we may consider $C_l := \lim_{j \ra \infty}C_{jl}$ in $A$ for each $l \in \bZ_{\geq 0}$ and $F(t) := C_0 + C_1t + C_2t^2 + \cdots A\llb t \rrb$. Since $\lim_{j \ra \infty}C_{jl} = 0$ in $A$, given any $k \in \bZ_{\geq 0}$, there exists some $m_k \in \bZ_{\geq 0}$ such that if $j > m_k$, then $C_{jl} \in p^{k}A$. As $C_l = \lim_{j \ra \infty}C_{jl}$, there exists some $n_k \in \bZ_{\geq 0}$ such that if $l > n_k$, then $C_l - C_{jl} \in p^kA$ so that $C_l \in p^kA$. This implies that $\lim_{l \ra \infty}C_l = 0$ so that $F(t) \in \Athat$. By definition of product topology on $A \llb t \rrb$, it follows that $\lim_{j \ra \infty}F_j(t) = F(t)$ in $A \llb t \rrb$. Hence, the last convergence also happens in $\Athat$. This finishes the proof.
\end{proof}

\begin{exmp}
    We have $(I_n - pI_n t)^{-1}=I_n +  pI_n t + p^2 I_n t^2 + \dots$ is an element of $\Athat$, while $(I_n - I_n t)^{-1}=I_n + I_n t + I_n t^2+\dots$ is not. 
\end{exmp}

\hspace{3mm} We are ready to state a main theorem of this section.

\begin{thm}[Noncommutative Weierstrass preparation theorem]\label{thm:weierstrass} Fix any $M(t), N(t)\in \Athat$. For any $X\in A$, there exists unique $U(t)\in \Athat$ and unique $X'\in A$ such that
    \begin{equation}\label{eq:preparation}
        (X + I_n t + p I_n t^2 M(t))\, U(t) = X' + I_n t + pI_n t^2 N(t).
    \end{equation}

    Moreover, we have $U(t)\in I_n + p \Athat$ and $X'\equiv X\ppmod p$.
\end{thm}

\begin{rmk} Theorem \ref{thm:weierstrass} can be generalized to a more general class of noncommutative rings, but we do not choose to do this in this paper for clarity. We also remark that any element in $I_n + p \Athat$ has a multiplicative inverse in $\Athat$, which can be seen by applying Lemma \ref{com}.
\end{rmk}

\hspace{3mm} We shall also need the version of the above theorem with $A_k := \M_n(\bZ/p^k\bZ)$ for arbitrary $k \in \bZ_{\geq 1}$ instead of $A$. We similarly define
\[\hhat{A_k[t]} : = \set*{\sum_{l=0}^{\infty}C_l t^l: C_l \in A_k \text{ and $\lim_{l \to \infty} C_l = 0$}},\]
but we are using the discrete topology on $A_k$, so having $\lim_{l \ra \infty} C_{l} = 0$ means that $C_l = 0$ for large enough $l$. This implies that $\hhat{A_k[t]} = A_k[t]$.

\begin{thm}[Finite noncommutative Weierstrass preparation theorem]\label{thm:finweierstrass} Fix any $M(t), N(t)\in A_k[t]$ for given $k \in \bZ_{\geq 1}$. For any $X \in A_k$, there exists unique $U(t) \in A_k[t]$ and unique $X' \in A_k$ such that
    \begin{equation}\label{eq:preparation}
        (X + I_n t + p I_n t^2 M(t))\, U(t) = X' + I_n t + pI_n t^2 N(t).
    \end{equation}

    Moreover, we have $U(t)\in I_n + p A_k[t]$ and $X'\equiv X\ppmod p$.
\end{thm}

\subsection{Proof that Theorems \ref{thm:weierstrass} and \ref{thm:finweierstrass} imply Lemma \ref{lem:equidist_map}}

Here we prove Lemma \ref{lem:equidist_map} assuming Theorems \ref{thm:weierstrass} and \ref{thm:finweierstrass}. Recall that, after this, the proof of Theorem \ref{main} would be complete once we prove Theorems \ref{thm:weierstrass} and \ref{thm:finweierstrass}.

\begin{proof}[Proof of Theorems \ref{thm:weierstrass} and \ref{thm:finweierstrass} imply Lemma \ref{lem:equidist_map}] Recall $R := \bZ_p[t]/(P(t))$. We first note that we can identify $A[t] = \M_n(\bZ_p[t])$. Consider the modulo-$(P(t))$ surjective map 
\[A[t] = \M_n(\bZ_p[t]) \tra \M_n(R).\] 
Explicitly, the map is given by
\[C_0 + C_1 t + C_2 t^2 + \cdots + C_m t^m \mapsto C_0 + C_1 \bar{t} + C_2 \bar{t}^2 + \cdots + C_m \bar{t}^m,\]
where $\bar{t}$ is the image of $t$ under the projection $\bZ_p[t] \tra \bZ_p[t]/(P(t))$. Now, consider any
\[F(t) = C_0 + C_1 t + C_2 t^2 + \cdots \in \Athat\]
Using the fact that $\lim_{l \ra\infty 0}C_l = 0$ in $A$ with the $p$-adic topology, given any $k \in \bZ_{\geq 1}$, there exists minimal $m_{F,k} \in \bZ_{\geq 1}$ such that if $l > m_{F,k}$, then $C_l \in p^{k}A$. This lets us define a map $\Athat \ra \M_n((\bZ/p^k\bZ)[t])$ given by
\[F(t) = \sum_{l=0}^{\infty} C_lt^l \mapsto \sum_{l=0}^{m_{F,k}}\bbar C_lt^l,\]
where $\bbar C_l$ is $C_l$ modulo $p^k$. Hence, we get a map $\Athat \ra \M_n((\bZ/p^k\bZ)[t]/(P(t)))$ given by
\[F(t) = \sum_{j=0}^{\infty}C_l t^l \mapsto \sum_{l=0}^{m_{F,k}}\bbar C_l \bar{t}^l.\]
Since $p^kA \sups p^{k+1}A \sups p^{k+2}A \sups \cdots$, we have $m_{F,k} \leq m_{F,k+1} \leq m_{F,k+2} \leq \cdots$. By taking $k=1$, we have $m_{F,1} \leq m_{F,2} \leq m_{F,3} \leq \cdots$, so this induces a map $\Athat \ra \M_n(\bZ_p[t]/(P(t))) = \M_n(R)$ compatible with the projection maps $\M_n((\bZ/p^{k+1}\bZ)[t]/(P(t))) \tra \M_n((\bZ/p^k\bZ)[t]/(P(t)))$ for all $k \geq 1$. We have $\sum_{j=0}^{\infty}C_j\bar{t}^j \in \M_n(R)$ as the image of $\sum_{j=0}^{\infty}C_jt^j \in \Athat$. This map is surjective because the map $A[t] \ra \M_n(R)$ we described above is surjective.
    
\hspace{3mm} Let $M(t)\in \Athat$ be any lift of $Y_2+\ol t Y_3 + \dots + \ol t^{d-3} Y_{d-1}\in \M_n(R)$ and fix $M(t)$ from now on. Then for any $X\in \M_n(\Zp)$, by Theorem \ref{thm:weierstrass} with $N(t)=0$, there exists a unique $U(t)\in I_n + p\Athat$ and $X'\in \M_n(\Zp)$ such that
\begin{equation}\label{eq:preparation_lifted}
    (X + I_nt + p t^2 M(t))\, U(t) = X'+ I_nt \in \Athat.
\end{equation}
Define the map $\Phi:\M_n(\Zp)\to \M_n(\Zp)$ by $\Phi(X):=X' = (X + I_nt + p t^2 M(t))\, U(t) - I_n t$. Theorem \ref{thm:weierstrass} implies $X\equiv X'\ppmod p$. We claim $\Phi$ is the desired bijection.

\hspace{3mm} First, we show $\Phi$ is a bijection by constructing an inverse. By switching the role of $M(t)$ and $N(t)$ in Theorem \ref{thm:weierstrass}, for any $X'\in \M_n(\Zp)$, there exists a unique $V(t)\in I_n + p\Athat$ and $X''\in \M_n(\Zp)$ such that
\begin{equation*}
    (X' + I_nt)\, V(t) = X''+ I_nt+ p t^2 M(t)\in \Athat.
\end{equation*}
Define the map $\Psi:\M_n(\Zp)\to \M_n(\Zp)$ by $\Psi(X'):=X'' = (X' + I_nt)\, V(t) - I_nt - p t^2 M(t)$. By the uniqueness statement in Theorem \ref{thm:weierstrass}, it follows that $\Psi$ is the inverse of $\Phi$. 

\hspace{3mm} Next, we note that \eqref{eq:preparation} holds for some $U\in \GL_n(R)$ instead of $U(t)$. This is immediate by letting $U$ be the image of $U(t)$ under $\Athat \tra \M_n(R)$, and applying this surjection to \eqref{eq:preparation_lifted}.

\hspace{3mm} Finally, we prove that $\Phi$ is Haar measure-preserving. It suffices to prove that for $k\geq 1$, the bijection $\Phi$ is compatible with the mod-$p^k$ reduction map. More precisely, we claim that if $X_1,X_2\in \M_n(\Zp)$ satisfy $X_1\equiv X_2\ppmod {p^k}$, then $\Phi(X_1)\equiv \Phi(X_2)\ppmod {p^k}$. To prove the claim, write $X_i'= \Phi(X_i)$ and $\bar X = (X_1\bmod {p^k})=(X_2\bmod {p^k})$. Theorem \ref{thm:finweierstrass} with $N(t) = 0$ implies that there exist unique $\bar X'\in A_k$ and $U(t)\in I_n + p A_k[t]$ such that
\begin{equation*}
    (\bar X + I_n t + \pi t^2 M(t))\, U(t)\equiv \bar X' + I_n t \ppmod {p^k}. 
\end{equation*}
If we replace $\bar X'$ in the above identity with $(X_1'\bmod {p^k})$ and $(X_2'\bmod {p^k})$, the new identity still holds by Theorem \ref{thm:weierstrass}. Hence, it follows from the uniqueness Theorem \ref{thm:finweierstrass} that $X_1'\equiv X_2'\ppmod {p^k}$.
\end{proof}

\hspace{3mm} For the rest of the section, we prove Theorems \ref{thm:weierstrass} and \ref{thm:finweierstrass}, which would finish the proof of Theorem \ref{main}. We start with some elementary observations.

\subsection{Elementary observations}

The following observation is simple but crucial in the proofs of Theorems \ref{thm:weierstrass} and \ref{thm:finweierstrass}. Recall the notation $A = \M_n(\bZ_p)$ and $A_k = \M_n(\bZ/p^k\bZ)$.

\begin{lem}\label{lem:terminate}
For any $k \in \bZ_{\geq 1}$, we can identify
\begin{equation*}
\frac{\Athat}{p^k \Athat} = (A/p^k A)[t] = A_k[t].
\end{equation*}
In other words, every element in $\Athat$ is a polynomial modulo $p^k$.
\end{lem}
\begin{proof}
This is simply because for any element $\sum_{l=0}^\infty C_l t^l \in \Athat$, we must have $\lim_{l \to \infty} C_l = 0$ with respect to the $p$-adic topology, so only finitely many $C_l$ are nonzero mod $p^k$. 
\end{proof}

\subsection{Uniqueness for Theorems \ref{thm:weierstrass} and \ref{thm:finweierstrass}} We now prove the uniqueness parts of Theorems \ref{thm:weierstrass} and \ref{thm:finweierstrass}:
\begin{proof}[Proofs of the uniqueness statements in Theorems \ref{thm:weierstrass} and \ref{thm:finweierstrass}] We first prove the uniqueness statement in Theorem \ref{thm:weierstrass}. Say
\begin{align*}
(X + I_nt + p t^2 M(t)) U_1(t) &= X_1' + I_n t + p t^2 N(t) \text{ and}\\
(X + I_nt + p t^2 M(t)) U_2(t) &= X_2' + I_n t + p t^2 N(t)
\end{align*}
are two expressions with $U_1(t), U_2(t)\in \Athat$ and $X_1',X_2'\in A$. Then
\begin{equation*}
(X + I_n t + p t^2 M(t)) f(t)=Y,
\end{equation*}
where $f(t) := U_1(t) - U_2(t)\in \Athat$ and $Y :=X_1' - X_2'\in A$. 

\hspace{3mm} We need to show that $f(t)=0$. To do so, it suffices to show $f(t) \equiv 0 \ppmod{p^k}$ for every $k \in \bZ_{\geq 0}$. We proceed by induction on $k$. The base case $k=0$ is vacuously true, and we assume $f(t)\equiv 0 \pmod{p^{k}}$ for arbitrary $k \in \bZ_{\geq 0}$. Reducing modulo $p^{k+1}$, we have
\begin{equation*}
Y = (X + I_n t+ p t^2 M(t)) f(t)\equiv (X + I_n t) f(t) \ppmod{p^{k+1}}.
\end{equation*}

For contradiction, suppose $f(t)\nequiv 0 \pmod{p^{k+1}}$. By Lemma \ref{lem:terminate}, the above identity can be considered in the \emph{polynomial ring} $(A/p^{k+1}A)[t] = A_{k+1}[t]$. In particular, $\bbar f(t):=f(t) \bmod p^{k+1}$ has a highest degree term because it is nonzero by assumption. Since the highest degree coefficient of $X+ I_n t$ is $I_n = 1_A$, which is not a zero divisor in $A$, the product $(X + I_n t) f(t)$ cannot be a constant modulo $p^{k+1}$. This contradicts with $(X +  I_n t)f(t)\equiv Y\ppmod{p^{k+1}}$, which completes the proof of the uniqueness statement of Theorem \ref{thm:weierstrass}.

\hspace{3mm} The proof of the uniqueness statement of Theorem \ref{thm:finweierstrass} is almost identical, so we omit it.
\end{proof}

\subsection{Proofs of final assertions in Theorems \ref{thm:weierstrass} and \ref{thm:finweierstrass}} Here, we prove that in either the setting of Theorem \ref{thm:weierstrass} or that of Theorem \ref{thm:finweierstrass}, if $U(t)$ and $X'$ in the statement exist, then they must satisfy $U(t)\equiv I_n \ppmod p$ and $X'\equiv X\ppmod p$. 

\begin{proof}[Proofs of final assertions in Theorems \ref{thm:weierstrass} and \ref{thm:finweierstrass}] We first assume Theorem \ref{thm:weierstrass} except its final assertion. Reducing \eqref{eq:preparation} modulo $p$ and using Lemma \ref{lem:terminate}, we have
    \begin{equation*}
    (\bar X + I_nt) \bar U(t) = \bar X' + I_n t \in (A/pA)[t],
    \end{equation*}
    where $\bar{X}$ denotes the reduction of $X$ modulo $p$ and similarly for $\bar X'$ and $\bar U(t)$. By comparing the highest degree terms of both sides, the only possibility for the above identity to hold in $(A/pA)[t]$ is when $\bar U(t) = I_n$. It then follows that $\bar X'=\bar X$.
    
  \hspace{3mm} The proof of the final assertion in Theorems \ref{thm:finweierstrass} is identical, so we omit it.
\end{proof}

\subsection{Proof of existence statements in Theorems \ref{thm:weierstrass} and \ref{thm:finweierstrass}}
Here, we prove the existence statements in Theorems \ref{thm:weierstrass} and \ref{thm:finweierstrass}. As is suggested by Example \ref{exmp:division}, our approach to constructing $U(t)$ and $X'$ is to perform a recursive algorithm and take the limit of the process. To be more systematic than the computations given in Example \ref{exmp:division}, we utilize the following division algorithm by the series $g(t):=X + I_n t+ p t^2 M(t)$.

\begin{lem}\label{lem:division} Fix $M(t)\in \Athat$. Define $g(t):=X + I_n t+ p t^2 M(t)$ and let $f(t)$ be any element of $\Athat$. Then there exist $q(t)\in \Athat$ and $r\in A$ such that
    \begin{equation*}
        f(t)=g(t) q(t)+r.
    \end{equation*}
\end{lem}

\hspace{3mm} Before proving Lemma \ref{lem:division}, we show why it would resolve Theorems \ref{thm:weierstrass} and \ref{thm:finweierstrass}, which would finish the proof of Theorem \ref{main}.

\begin{proof}[Proof of the existence statement of Theorems \ref{thm:weierstrass} and \ref{thm:finweierstrass} assuming Lemma \ref{lem:division}]
    Construct $q(t)$ and $r$ using Lemma \ref{lem:division} with $f(t) := I_n t + p t^2 N(t)$. Letting $U(t)=q(t)$ and $X'=-r$, this proves Theorem \ref{thm:weierstrass}. For Theorem \ref{thm:finweierstrass}, we reduce the statement of Lemma \ref{lem:division} modulo $p^k$ and then repeat the proof.
\end{proof}

\hspace{3mm} Hence, it remains to show Lemma \ref{lem:division} to prove Theorem \ref{main}. Given $f(t)$ and $g(t)$ as in Lemma \ref{lem:division}, we describe an algorithm to construct sequences $(q_j(t))_{j \geq 1}$ and $(r_j(t))_{j \geq 1}$, and prove that they converge to the desired elements $q(t)$ and $r$, respectively. More precisely, we prove the following lemma, which is stronger than Lemma \ref{lem:division}.

\begin{lem} Assume the hypotheses of Lemma \ref{lem:division}. Define $q_1(t) := 0$ and $r_1(t) := f(t)$ and recursively construct $q_j(t)$ and $r_j(t)$ for $j \geq 1$ by
    \begin{equation}\label{eq:division_algorithm}
    \left\{
    \begin{array}{l}
    q_{j+1}(t)=q_j(t)+\dfrac{s_j(t)}{t},\\
    r_{j+1}(t)=r_j(t) - g(t) \dfrac{s_j(t)}{t},
    \end{array}
    \right.
    \end{equation}
    where $s_j(t) := r_j(t) - r_j(0)$, which is the sum of all nonconstant terms of $r_j(t)$. Then both $(q_j(t))_{j \in \bZ_{\geq 1}}$ and $(r_j(t))_{j \in \bZ_{\geq 1}}$ converge $p$-adically in $\Athat$. Moreover, if $q(t) := \lim_{j\to \infty} q_j(t)$ and $r(t) := \lim_{j \to \infty} r_j(t)$, then $r(t) = r \in A$ and $f(t) = g(t)q(t)+r$.
\end{lem}
\begin{proof}
    We note by the recursive construction \eqref{eq:division_algorithm} that we always have 
    \begin{equation}\label{eq:division_step}
        f(t)=g(t)q_j(t)+r_j(t)
    \end{equation}
    for all $j \in \bZ_{\geq 1}$. To prove the convergence of sequences $(q_j(t))_{j \in \bZ_{\geq 1}}$ and $(r_j(t))_{j \in \bZ_{\geq 1}}$ in $\Athat$, we work modulo $p^k$ for any given $k \geq 1$. We again use the notation $A_k = A/p^k A$ and note that $\Athat/p^k \Athat = A_k[t]$ by Lemma \ref{lem:terminate}. We denote by $\ol{q_j(t)}$ the image of $q_j(t)$ in $A_k[t]$, and similarly for $\ol{r_j(t)}$ and $\ol{s_j(t)}$. We claim that for any $k \in \bZ_{\geq 1}$, we have 
    \begin{equation}\label{eq:division_stabilize}
        \ol{s_j(t)} = 0 \in A_k[t] \text{ for large enough } j \geq 1.
    \end{equation}
    
\hspace{3mm} Before we prove \eqref{eq:division_stabilize}, we note that proving this claim suffices to prove the desired result. Indeed, if $\ol{s_j(t)}$ is eventually zero, then $\ol{q_j(t)}$ and $\ol{r_j(t)}$ eventually stabilizes from \eqref{eq:division_algorithm}. Since this is true for arbitrary $k \geq 1$, both $(q_j(t))_{j \in \bZ_{\geq 1}}$ and $(r_j(t))_{j \in \bZ_{\geq 1}}$ converge in $\Athat$ because $\Athat$ is $p$-adically complete by Lemma \ref{com}. We denote their limits by $q(t)$ and $r(t)$, and we have $f(t)=g(t)q(t)+r(t)$ by taking the $p$-adic limit of \eqref{eq:division_step} as $j \to \infty$. Furthermore, it follows from the definition of $s_j(t)$ that $\ol{r_j(t)} = \ol{r_j(0)} \in A_k[t]$ for large enough $j \geq 1$ given arbitrary $k$, so we must have $\lim_{j \ra \infty}(r_j(t) - r_j(0)) = 0$ in $\Athat$, which implies that
\[\lim_{j \ra \infty}r_j(0) = \lim_{j \ra \infty}(r_j(t) - (r_j(t) - r_j(0))) = r(t)\]
in $\Athat$. This implies that $r(t) \in A$.

\hspace{3mm} We now prove \eqref{eq:division_stabilize}. As we work in $A_k[t]$, we denote by $M(t), f(t), g(t), q_j(t), r_j(t), s_j(t)$ to mean their reductions modulo $p^k$. Let $D\geq 1$ be the degree of $g(t) = X + I_n t+ p t^2 M(t)$ as a \emph{polynomial} in $A_k[t]$. Fix a real number $\epsilon$ such that $0<\epsilon\leq 1/D$. For a monomial $at^b$ in $A_k[t]$ with nonzero $a \in A_k$ and $b\geq 0$, we define
    \begin{equation*}
        \delta(a t^b) := v_k(a) - \epsilon b \in \bR,
    \end{equation*}
where $v_k(a) := \max\{m \in \bZ_{\geq 0} : a \in p^m A_k\}$. Since $a \neq 0$, we have $v_k(a) \in \{0, 1, \dots, k-1\}$. For example, we have $\delta(I_n t)=-\eps$ and $\delta(a)=v_k(a)\geq 0$ for any nonzero $a \in A_k$. We also define $\dt(0) := \infty$. More generally, for any polynomial $f(t)\in A_k[t]$, we define $\delta(f)$ to be the minimal $\delta$-valuation of terms of $f(t)$. Note that $\delta(f)=\infty$ if and only if $f(t)=0$ in $A_k[t]$. Thus, our goal is to show that $\dt(s_j(t)) = \infty$ for large enough $j \geq 1$.

\hspace{3mm}  Since $p t^2 M(t) = g(t) - X - I_n t$, we see that $t^2 M(t)$ has degree at most $D$. Since there is no constant term for $p t^2 M(t)$, we have
    \begin{equation*}
        \delta(p t^2 M(t)) \geq 1 - \eps D \geq 0.
    \end{equation*}
    
\hspace{3mm} We claim $\lim_{j \to \infty} \delta(s_j(t))=\infty$. This is the crux of the entire proof. Expand \eqref{eq:division_algorithm} to get
    \begin{equation}\label{rel}
    r_{j+1}(t) = r_j(t) - s_j(t) -  X \frac{ s_j(t)}{t} - p I_n t  M(t)  s_j(t)
    \end{equation}
    and inspect the $\delta$-valuations of its terms. If we denote by $a t^b$ is a typical term of $s_j(t)$ with nonzero $a \in A_k$ and $b \geq 1$, a typical term for $s_j(t)/t$ can be described as $a t^{b-1}$. If $Xa = 0$, then $\dt(Xa t^{b-1}) = \infty$. Otherwise, we have
\[\dt(X a t^{b-1}) = v_k(Xa) - \ep(b-1) = v_k(Xa) - \ep b + \ep \geq v_k(a) - \ep b + \ep = \dt(a t^b) + \ep,\]
so we always have   
    \begin{equation*}
    \delta \lt( X \frac{ s_j(t)}{t} \rt) \geq \delta( s_j(t))+\eps.
    \end{equation*}

\hspace{3mm} We note $\delta(f_1(t)f_2(t))\geq \delta(f_1(t))+\delta(f_2(t))$ for any $f_1(t), f_2(t) \in A_k[t]$ from definition of $\dt$. Since $t M(t)$ has degree at most $D-1$ and has no constant term, we have
    \begin{equation*}
    \delta(p I_n t  M(t)  s_j(t)) \geq \delta(p I_n t  M(t)) + \delta(s_j(t)) \geq 1-(D-1)\eps+\delta( s_j(t) ) \geq \delta( s_j(t) )+\eps
    \end{equation*}
    by our assumption that $\eps \leq 1/D$.
        
\hspace{3mm} Since $r_j(t) - s_j(t) = - r_j(0)$ has only constant term, every possible nonconstant term of $r_{j+1}(t)$ in \eqref{rel} must be contributed from $X { s_j(t)}/{t}$ and $pI_n t  M(t)  s_j(t)$. Since
\[s_{j+1}(t) = r_{j+1}(t) - r_{j+1}(0) = -r_{j+1}(0) + r_j(0) -  X \frac{ s_j(t)}{t} - p I_n t  M(t)  s_j(t),\]
using the fact that $s_{j+1}(t)$ has no constant terms, we have
\begin{align*}
    \delta( s_{j+1}(t) ) &\geq  \min\lt\{\dt\lt(X \frac{ s_j(t)}{t}\rt),  \dt\lt(p I_n t  M(t)  s_j(t)\rt)\rt\} \\
    &\geq \delta( s_j(t))+\eps.
\end{align*}
    
In particular, we have $\lim_{j\to \infty} \delta( s_j(t))= \infty$, but the largest possible finite $\delta$-value in $A_k[t]$ is $k-1$: since $p^k A_k=0$, the largest possible finite $v_k(a)$ is $k-1$, so $\delta(at^b) = v_k(a) - \eps b\leq k-1$ for any nonzero monomial $at^b\in A[t]$. Hence, $\delta(s_j(t))=\infty$ for $j \gg 1$, which implies \eqref{eq:division_stabilize}. 
\end{proof}

\hspace{3mm} We are done with proving Theorem \ref{main}. For the rest of the paper, we use Theorem \ref{main} to prove the remaining parts of Theorem \ref{main2}.

%%%%%%%%%%%%%%%%%%%%%

\section{Reduction of Theorem \ref{main2} in terms of moments}\label{red}

\hspace{3mm} By choosing any $k \in \bZ_{\geq 1}$ such that $p^{k-1}G = 0$, Theorem \ref{main2} can be proven by proving the analogous statement we get by replacing $\bZ_p$ with $\bZ/p^k\bZ$. (The details can be found in \cite[Lemmas 2.1 and 3.1]{CY}.) Write $R := (\bZ/p^k\bZ)[t]/(P(t))$ for the rest of the paper. Fix $n \in \bZ_{\geq 1}$, and we assume that $A_n \in \M_n(\bF_p)$ is of the form \eqref{specialA}:
\[A_n = \begin{bmatrix}
J & \ast \\
0 & J'
\end{bmatrix},\]
where $J \in \M_{n- r}(\bF_p)$ and $J' \in \M_r(\bF_p)$ with $r = r_p(G)$ such that every eigenvalue of $J$ in $\ol{\bF_p}$ is not a root of $\bar{P}(t)$. We fix a finite-sized $\bF_p[t]/(\bar{P}(t))$-module $\rr$ so that $\rr \simeq_{\bF_p[t]} G/pG \simeq_{\bF_p[t]} \cok(\bar{P}(A_n))$. We introduce this notation because we may vary $G$, while the isomorphism class of $G/pG$ is fixed (as a $\bF_p[t]/(P(t))$-module). We shall write
\[\M_n(\bZ/p^k\bZ)_{A_n} := \{X \in \M_n(\bZ/p^k\bZ) : X \equiv A_n \ppmod{p}\}\]
so that 
\[\underset{X \in \M_n(\bZ/p^k\bZ)}{\Prob}(\cok(P(X)) \simeq_R G \mid X \equiv A_n \ppmod{p}) = \underset{X \in \M_n(\bZ/p^k\bZ)_{A_n}}{\Prob}(\cok(P(X)) \simeq_R G).\]
That is, we consider $\M_n(\bZ/p^k\bZ)_{A_n}$ as the sample space instead of mentioning conditional probabilities for the statement of Theorem \ref{main2} (after we replace $\bZ_p$ by $\bZ/p^k\bZ$). The \textbf{Haar measure} on $\M_n(\bZ/p^k\bZ)_{A_n}$ is defined to be the probability measure induced by the Haar measure of $\M_n(\bZ/p^k\bZ)$, which is equal to the uniform measure. If $k = 1$, the statement we get from replacing $\bZ_p$ with $\bZ/p^k\bZ$ in Theorem \ref{main2} is immediate (as $p^{k-1}G = 0$ with $k=1$ would imply $G = 0$), so we may assume $k \geq 2$ from now on. Given $X \in \M_n(\bZ/p^k\bZ)$, its $(i,j)$-entry $X_{ij}$ can be written as
\begin{equation}\label{digits2}
X_{ij} = X_{i,j,0} + X_{i,j,1} p + X_{i,j,2} p^2 + \cdots + X_{i,j,k-1} p^{k-1}
\end{equation}
with $X_{i,j,l} \in \{0, 1, 2, \dots, p-1\}$. When $X \in \M_n(R)_{A_n}$, we have $X_{i,j,0} = A_{ij}^{(n)}$ fixed, where $A_{ij}^{(n)}$ is the $(i,j)$-entry of $A_n$. Having $X \in \M_n(\bZ/p^k\bZ)_{A_n}$ follow the Haar measure is equivalent to having $X_{i,j,0} = A_{ij}^{(n)}$ and $X_{i,j,1}, X_{i,j,2}, \dots, X_{i,j,k-1}$ uniformly distributed in $\{0,1,2, \dots, p-1\}$. We work with the discrete $\sg$-algebra on $\M_n(\bZ/p^k\bZ)_{A_n}$, and we assume that $X \in \M_n(\bZ/p^k\bZ)$ has $n^{2}$ independent entries and that the entries of the bottom-right $r \times r$ submatrix of $X$ are uniformly distributed, where $r = \dim_{\bF_p}(\rr)$.

\hspace{3mm} Denote by $\Mod_A^{<\infty}$ the set of isomorphism classes of finite size $A$-modules for a given commutative ring $A$. Given $H \in \Mod_R^{<\infty}$, the \textbf{$H$-moment} of the distribution $(\cok(P(X)))_{X \in \M_n(\bZ/p^k\bZ)}$ is defined to be 
\[\underset{X \in \M_n(\bZ/p^k\bZ)_{A_n}}{\bE}(|\Sur_R(\cok(P(X)), H)|),\]
where $\Sur_R(S, T)$ means the set of surjective $R$-linear maps from $S$ to $T$ given $S, T \in \Mod_R^{<\infty}$. Sawin and Wood \cite[Lemma 6.1]{SW} noticed that the category of finite size $R$-modules is a \textbf{diamond category}, whose definition can be found in \cite[Definition 1.3]{SW}. The point of working in a diamond category is that the $H$-moments of a distribution in such a category determines the distribution, where $H$ varies in the category, as long as the $H$-moments do not ``grow too fast'' (i.e., the $H$-moments are \textbf{well-behaved} in the sense of \cite[p.4]{SW}).

\subsection{The Haar moment is independent to $n$}\label{ind} By applying Theorem \ref{main}, when $\M_n(\bZ/p^k\bZ)_{A_n}$ is given the Haar measure, the $H$-moment of the distribution $(\cok(P(X)))_{X \in \M_n(\bZ/p^k\bZ)}$ is
\begin{align*}
&\underset{X \in \M_n(\bZ/p^k\bZ)_{A_n}^{\Ha}}{\bE}(|\Sur_R(\cok(P(X)), H)|) \\
&= \sum_{M \in \Mod_R^{<\infty}} |\Sur_R(M, H)| \underset{X \in \M_n(\bZ/p^k\bZ)_{A_n}^{\Ha}}{\Prob}(\cok(P(X)) \simeq_R M) \\
&= \sum_{M \in \Mod_R^{<\infty}} |\Sur_R(M, H)| \underset{Y \in \M_n(\bZ_p)_{A_n}^{\Ha}}{\Prob}((\cok(P(Y)) \otimes_{\bZ_p} \bZ/p^k\bZ) \simeq_R M) \\
&= \sum_{M \in \Mod_R^{<\infty}} \sum_{\substack{W \in \Mod_{\bZ_p[t]/(P(t))}^{<\infty}: \\ W \otimes_{\bZ_p} \bZ/p^k\bZ \simeq_R M}} |\Sur_R(W \ot_{\bZ_p} \bZ/p^k\bZ, H)| \underset{Y \in \M_n(\bZ_p)_{A_n}^{\Ha}}{\Prob}(\cok(P(Y)) \simeq_R W) \\
&= \sum_{W \in \Mod_{\bZ_p[t]/(P(t))}^{<\infty}} |\Sur_R(W \ot_{\bZ_p} \bZ/p^k\bZ, H)| \underset{Y \in \M_n(\bZ_p)_{A_n}^{\Ha}}{\Prob}(\cok(P(Y)) \simeq_R W) \\
&= \sum_{\substack{W \in \Mod_{\bZ_p[t]/(P(t))}^{<\infty}: \\ W/pW \simeq_{\bF_p[t]} \rr}} |\Sur_R(W \ot_{\bZ_p} \bZ/p^k\bZ, H)| \underset{Y \in \M_n(\bZ_p)_{A_n}^{\Ha}}{\Prob}(\cok(P(Y)) \simeq_R W) \\
&= \sum_{\substack{W \in \Mod_{\bZ_p[t]/(P(t))}^{<\infty}: \\ W/pW \simeq_{\bF_p[t]} \rr \text{ and} \\ |\Hom_{\bZ_{p}[t]}(W, \bF_{p^{d_{j}}})| = |\Ext_{\bZ_{p}[t]/(P(t))}^{1}(W, \bF_{p^{d_{j}}})| \\ \text{for } 1 \leq j \leq l }} |\Sur_R(W \ot_{\bZ_p} \bZ/p^k\bZ, H)| \frac{|\Aut_{\bZ_p[t]}(W/pW)|\prod_{j=1}^{l}\prod_{i=1}^{u_j(\rr)}(1 - p^{-id_j})}{|\Aut_{\bZ_p[t]}(W)|}.
\end{align*}
The last sum is a convoluted expression, but we can still observe that this only depends on $p, k, P(t), \rr,$ and $H$, not depending on $A_n$ nor $n$. Since we fix $p, k, P(t),$ and $\rr$, this justifies the following notation:
\[M_H := \underset{X \in \M_n(\bZ/p^k\bZ)_{A_n}^{\Ha}}{\bE}(|\Sur_R(\cok(P(X)), H)|).\]

\subsection{The Haar moment is well-behaved}\label{well-behaved} We have
\[M_H =\underset{X \in \M_n(\bZ/p^k\bZ)_{A_n}^{\Ha}}{\bE}(|\Sur_R(\cok(P(X)), H)|) = \sum_{\substack{M \in \Mod_R^{<\infty}: \\ M/pM \simeq_{\bF_p[t]} \rr}} |\Sur_R(M, H)| \underset{X \in \M_n(\bZ/p^k\bZ)_{A_n}^{\Ha}}{\Prob}(\cok(P(X)) \simeq_R M),\]
which is bounded above by
\[\sum_{\substack{M \in \Mod_R^{<\infty}: \\ M/pM \simeq_{\bF_p[t]} \rr}} |\Hom_R(M, H)| \leq C_{\rr} |H|^{N_{\rr}}\]
for some constants $C_{\rr}, N_{\rr} > 0$ depending only on $\rr$. We explain how the last inequality holds. First, note that by Hensel's lemma, we have a factorization
\[P(t) = Q_1(t)Q_2(t) \cdots Q_l(t) \in (\bZ/p^k\bZ)[t]\]
such that each $Q_j(t)$ is a monic polynomial whose reduction modulo $p$ is $\bar{Q}_j(t) = \bar{P}_j(t)^{m_j}$ in $\bF_p[t]$. These $Q_1(t), Q_2(t), \dots, Q_l(t)$ are pairwise comaximal in $(\bZ/p^k\bZ)[t]$, so we have $R \simeq R_1 \times R_2 \times \cdots \times R_l$ as rings with $R_j := (\bZ/p^k\bZ)[t]/(Q_j(t))$ by the Chinese Remainder Theorem. If we consider any $M$ in the last summand, this necessarily implies that $M \simeq_R M_1 \times M_2 \times \cdots \times M_l$, where each $M_j$ is an $R_j$-module, and this implies 
\[\rr \simeq_{\bF_p[t]} M/pM \simeq_{\bF_p[t]} (M_1/pM_1) \times (M_2/pM_2) \times \cdots \times (M_l/pM_l).\]
Since each $R_j$ is a local ring with the maximal ideal $(p, P_j(t))$ where $P_j(t) \in (\bZ/p^k\bZ)[t]$ is a lift of $\bar{P}_j(t) \in \bF_p[t]$, Nakayama's lemma implies that $M_j$ can be generated by $|\Hom_{\bF_p[t]}(\rr, \bF_{p^{d_j}})|$ elements. Thus, taking $N_{\rr} := \sum_{j=1}^l|\Hom_{\bF_p[t]}(\rr, \bF_{p^{d_j}})|$ and $C_{\rr}$ to be the number of $M \in \Mod_{R}^{< \infty}$ such that $M/pM \simeq_{\bF_p[t]} \rr$, we establish the desired inequality.

\subsection{Reduction of Theorem \ref{main2} in terms of moments} By \cite[Corollary 6.5]{SW}, the previous subsection shows that $(M_H)_{H \in \Mod_{R}^{<\infty}}$ are well-behaved, so we may apply \cite[Theorem 1.6]{SW} to reduce the problem of showing the rest of Theorem \ref{main2} (in addition to Theorem \ref{main} that we previously established) into the problem of showing that every $H$-moment for the distribution $(\cok(P(X)))_{X \in \M_n(\bZ/p^k\bZ)_{A_n}}$ is equal to $M_H$. Thus, applying Lee's linearization trick \eqref{Lee}, proving Theorem \ref{main2} is reduced into proving the following:

\begin{thm}\label{main3} Suppose that $(\M_n(\bZ/p^k\bZ)_{A_n})_{n \in \bZ_{\geq 1}}$ are given probability measures such that each random $X \in \M_n(\bZ/p^k\bZ)$ has $n^2$ independent entries. If $A_n$ is of the form \eqref{specialA} and the entries of the bottom-right $r \times r$ submatirx of $X$ are uniformly distributed with $r = \dim_{\bF_p}(\rr)$, then
\[\underset{X \in \M_n(\bZ/p^k\bZ)_{A_n}}{\bE}(|\Sur_R(\cok_R(X - \bar{t}I_n), H)|) = M_H\]
for every $H \in \Mod_{R}^{<\infty}$.
\end{thm}

\section{Proof of Theorem \ref{main3}} 

\hspace{3mm} For the rest of the paper, we prove Theorem \ref{main3}. Fix $H \in \Mod_{R}^{<\infty}$. Denoting by $\mu_n$ the given measure on $\M_n(\bZ/p^k\bZ)_{A_n}$ and $\bb{1}(\ms{P})$ the characteristic function of a property $\ms{P}$, we have
\begin{align*}
\underset{X \in \M_n(\bZ/p^k\bZ)_{A_n}}{\bE}(|\Sur_R(\cok_R(X - \bar{t}I_n), H)|) &= \int_{X \in \M_n(\bZ/p^k\bZ)_{A_n}} |\Sur_R(\cok_R(X - \bar{t}I_n), H)| d \mu_n \\
&= \int_{X \in \M_n(\bZ/p^k\bZ)_{A_n}} \sum_{\bar{F} \in \Sur_R(\cok_R(X - \bar{t}I_n), H)} 1 d \mu_n \\
&= \int_{X \in \M_n(\bZ/p^k\bZ)_{A_n}} \sum_{F \in \Sur_R(R^n, H)} \bb{1}(F(X - \bar{t}I_n) = 0) d \mu_n \\
&= \sum_{F \in \Sur_R(R^n, H)}\underset{X \in \M_n(\bZ/p^k\bZ)_{A_n}}\Prob(F(X - \bar{t}I_n) = 0).
\end{align*}

\hspace{3mm} We first note that for many $F \in \Sur_R(R^n, H)$, the summand in the last sum is $0$. We have
\begin{align*}
\underset{X \in \M_n(\bZ/p^k\bZ)_{A_n}}{\Prob}(F(X - \bar{t}I_n) = 0) &= \underset{B \in \M_n(\bZ/p^k\bZ)}{\Prob}(F(A_n + pB - \bar{t}I_n) = 0) \\
&= \underset{B \in \M_n(\bZ/p^k\bZ)}{\Prob}(pFB = -F(A_n - \bar{t}I_n)),
\end{align*}
where the entries of $B \in \M_n(\bZ/p^k\bZ)$ are independent and the entries in the bottom-right $r \times r$ submatrix of $B$ are uniformly distributed, where $r = \dim_{\bF_p}(\rr)$. We note that the above probability is $0$ when the image of $F(A_n - \bar{t}I_n)$ is not in $pH$. We shall identify
\[\Hom_R(R^n, pH) = \{\phi \in \Hom_R(R^n, H) : \im(\phi) \sub pH\}.\] 

\begin{notn}\label{resdef} From now on, we write
\bi
	\item $\Hom_{R}(R^n, H)_{A_n} := \{F \in \Hom_R(R^n, H) : F(A_n - \bar{t}I_n) \in \Hom_R(R^n, pH)\}$ and
	\item $\Sur_{R}(R^n, H)_{A_n} := \{F \in \Sur_R(R^n, H) : F(A_n - \bar{t}I_n) \in \Hom_R(R^n, pH)\}$.
\ei
\end{notn}

Moreover, we also note that the condition $F(X - \bar{t}I_n) = 0$ implies that $F(\bar{t}v) = F(Xv) \in F((\bZ/p^k\bZ)^{n})$ for any $v \in (\bZ/p^k\bZ)^n$. In particular, for any such $F$, we have $F((\bZ/p^k\bZ)^{n}) = F(R^n)$.

\begin{notn}\label{resdef2} We write
\bi
	\item $\Hom_{R}(R^n, H)_{A_n}^{\#} := \{F \in \Hom_R(R^n, H)_{A_n} : F((\bZ/p^k\bZ)^n) = F(R^n)\}$ and
	\item $\Sur_{R}(R^n, H)_{A_n}^{\#} := \{F \in \Hom_R(R^n, H)_{A_n}^{\#} : F \text{ is surjective}\}$.
\ei
\end{notn}

We note that to show Theorem \ref{main3}, it suffices to show
\begin{equation}\label{suff}
\sum_{F \in \Sur_{R}(R^n, H)_{A_n}^{\#}}\lt(\underset{X \in \M_n(\bZ/p^k\bZ)_{A_n}}\Prob(F(X - \bar{t}I_n) = 0) - \underset{X \in \M_n(\bZ/p^k\bZ)_{A_n}^{\Ha}}\Prob(F(X - \bar{t}I_n) = 0)\rt) = 0.
\end{equation}

\hspace{3mm} The following lemma counts $\#\Sur_{R}(R^n, H)_{A_n}$, which is an upper bound of $\#\Sur_{R}(R^n, H)_{A_n}^{\#}$.

\begin{lem}\label{res} We have
\be
	\item $\#\Hom_{R}(R^n, H)_{A_n} = \#\Hom_{R}(\rr, H/pH) |pH|^n$ and
	\item $\#\Sur_{R}(R^n, H)_{A_n} = \#\Sur_{R}(\rr, H/pH) |pH|^n$.
\ee
\end{lem}

\begin{proof} Write $Y := A_n - \bar{t}I_n \in \M_n(R)$ and denote by $\bar{Y} \in \M_n(R/pR)$ the reduction of $Y$ modulo $p$. For any $F \in \Hom_R(R^n, H)$, denoting by $\bar{F}$ its reduction modulo $p$, we see that $FY \in \Hom_R(R^n, pH)$ if and only if $\bar{F}\bar{Y} = 0 \in \Hom_{R/pR}((R/pR)^n, H/pH)$. Since $\rr \simeq_{\bF_p[t]} \cok(\bar{P}(A_n)) \simeq_{\bF_p[t]} \cok(A_n - \bar{t}I_n) = \cok(\bar{Y})$, the number of $\bar{F}$ such that $\bar{F}\bar{Y} = 0$ is 
\[\#\Hom_{R}(\cok(\bar{Y}), H/pH) = \#\Hom_{R/pR}(\rr, H/pH).\]
Since the size of each fiber under the modulo $p$ projection
\[\Hom_{R}(R^n, H) \tra \Hom_{R/pR}((R/pR)^n, H/pH)\]
is $\#\Hom_R(R^n, pH) = |pH|^n$, this finishes the proof of (1). The same proof works for (2) because $F$ is surjective if and only if $\bar{F}$ is.
\end{proof}

\begin{notn} From now on, we write $V := R^n$ and $V' := (\bZ/p^k\bZ)^n$ for convenience although both expressions do depend on $n$. We write $v_1, \dots, v_n$ to mean the standard $R$-basis for $V$. The same notation also means the standard $\bZ/p^k\bZ$-basis for $V'$.
\end{notn}

\subsection{Deterministic property of each $F$ and proof of Theorem \ref{main2}} We fix any $F \in \Sur_{R}(R^n, H)_{A_n}^{\#}$. Recall that $F$ satisfies $F(V') = F(V) = H$. Denoting by $\bar{F} : (R/pR)^n \ra H/pH$ the surjective map induced by $F$, we also note that its restriction $\bF_p^n \ra H/pH$ is a surjective $\bF_p$-linear map. We denote by $h := r_p(H)$ the $\bF_p$-dimension of $H/pH$. We may assume that $r = \dim_{\bF_p}(\rr) \geq h$ because otherwise \eqref{suff} holds trivially. Recall that $A_n$ is of the form \eqref{specialA}, and since $J \in \M_{n-r}(\bF_p)$ does not have any eigenvalues that are roots of $P(t)$ over $\ol{\bF_p}$, we know that $J - \bar{t}I_{n-r} \in \M_{n-r}(\bF_p[t]/(\bar{P}(t)))$ is invertible because its image over $\bF_p[t]/(\bar{P}_j(t))$ is invertible for all $1 \leq j \leq l$. Since $\bar{F}(A_n - \bar{t}I_n) = 0$, due to the form \eqref{specialA}, we must have $\bar{F}|_{(R/pR)^{n-r}}(J - \bar{t}I_{n-r}) = 0$, so the invertibility of $J - \bar{t}I_{n-r}$ implies that $\bar{F}|_{(R/pR)^{n-r}} = 0$, which is equivalent to saying that $F(v_1), \dots, F(v_{n-r}) \in pH$. Applying Nakayama's lemma, this implies that $F(v_{n-r+1}), \dots, F(v_n)$ generate $H$.

\begin{proof}[Proof of Theorem \ref{main2}]  We may consider a random matrix $X \in \M_n(\bZ/p^k\bZ)_{A_n}$ by writing $X = A_n + pB$, where $B$ is a random matrix in $\M_n(\bZ/p^k\bZ)$. Having $F(X - \bar{t}I_n) = 0$ is equivalent to $F(A - \bar{t}I_n) = pFB$, which can be seen as a system of equations
\[F(A_n - \bar{t}I_n)v_j = \sum_{i=1}^{n}pB_{ij}F(v_i),\]
for $1 \leq j \leq n$, where $B_{ij}$ is the $(i,j)$-entry of $B$. Due to the form \eqref{specialA}, we know that $(A_n - \bar{t}I_n)v_1, \dots, (A_n - \bar{t}I_n)v_{n-r}$ form an $R$-basis for $R^{n-r}$, so choosing values for $F(v_1), \dots, F(v_{n-r})$ is equivalent to choosing values of $F(A_n - \bar{t}I_n)v_1, \dots, F(A_n - \bar{t}I_n)v_{n-r}$. We may rewrite each equation as 
\[F(A_n - \bar{t}I_n)v_j - \sum_{i=1}^{n-r}pB_{ij}F(v_i)= \sum_{i=n-r+1}^{n}pB_{ij}F(v_i),\]
so considering $1 \leq j \leq n-r$, we see that any choice of $F(v_{n-r+1}), \dots, F(v_n) \in H$ and the entries of $B$ that are not in the bottom-right $r \times r$ submatrix of $B$ determine $F(v_1), \dots, F(v_{n-r}) \in pH$. We also note that such choices of entries of $B$ have no constraints. Hence, we see that the probability that $F(X - \bar{t}I_n) = 0$ is completely determined by the values of $F(v_{n-r+1}), \dots, F(v_n)$ and the entries of $r \times r$ bottom-right submatrix of $B$. This implies that we have
\[\underset{X \in \M_n(\bZ/p^k\bZ)_{A_n}}\Prob(F(X - \bar{t}I_n) = 0) = \underset{X \in \M_n(\bZ/p^k\bZ)_{A_n}^{\Ha}}\Prob(F(X - \bar{t}I_n) = 0),\]
so we must have \eqref{suff}, which implies Theorem \ref{main2}.
\end{proof}

\section*{Acknowledgments}

\hspace{3mm} We thank Nathan Kaplan for helpful discussions and comments on an earlier draft of this paper. We thank Rohan Das, Christopher Qiu, and Shiqiao Zhang for sharing some computer generated data relevant to the paper. We thank Melanie Matchett Wood for helpful advice for the last part of this paper. The first author received support from NSF grant DMS 2154223 for the project. The second author thanks the AMS-Simons Travel Grant for supporting his visit to the first author. The first author thanks Jungin Lee, Youn-Seo Choi, and the Korea Institute for Advanced Study for their hospitality during his visit to the institute, thanks Myungjun Yu and Yeonsei University for their hospitality during his visit to the university, and also thanks Peter Jaehyun Cho and Ulsan National Institute of Science and Technology for their hospitality during his visit to a workshop, where a part of this work was completed.

\end{document}